\documentclass[10pt]{article}
\usepackage{tikz,tikz-3dplot}
\usepackage{amsmath,amsthm}
\usepackage{amssymb}
\usepackage{amsfonts}
\usepackage{xcolor}
\usepackage{stmaryrd}
\usepackage{graphicx}
\usepackage{sidecap}
\usepackage[T1]{fontenc}
\usepackage{microtype}
\usepackage{floatrow}
\usepackage[english]{babel}
\usepackage[hmarginratio=1:1,top=30mm,columnsep=20pt]{geometry}
\usepackage[hang, small,labelfont=bf,up,textfont=it,up]{caption}
\usepackage{booktabs}
\usepackage{enumitem}
\usepackage{indentfirst}
\usepackage{lettrine}
\usepackage{subfigure}
\usepackage{diagbox}
\usepackage{multirow}
\usepackage{multicol}
\setlist[itemize]{noitemsep}
\usepackage{titlesec}
\usepackage[ruled]{algorithm2e}
\usepackage{titling}
\usepackage{hyperref}
\usepackage{tkz-euclide,pgfplots}

\usepackage{float} 
\usepackage{tkz-euclide,pgfplots}
\usepackage{amsthm}
\usepackage{amsmath}
\usepackage{amssymb}
\usepackage{indentfirst}
\usepackage{lettrine}
\usepackage{booktabs} 
\usepackage{hyperref}
\usepackage{multirow} 
\usepackage{multicol} 
\usepackage{arydshln}
\usepackage{verbatim} 
\usepackage{tkz-euclide,pgfplots}
\newtheorem{lemma}{Lemma}
\newtheorem{theorem}{Theorem}

\newtheorem{remark}{Remark}

\title{Two Results on Low-Rank Heavy-Tailed Multiresponse Regressions}
\author{Kangqiang Li\thanks{Corresponding author E-mail address: 11935023@zju.edu.cn (Kangqiang Li)} \qquad Yuxuan Wang\thanks{E-mail address: 12235030@zju.edu.cn (Yuxuan Wang)}\\[1ex] 
	\normalsize School of Mathematical Sciences, Zhejiang University,  Hangzhou,  Zhejiang 310027,  China \\ 
}
\date{}
\begin{document}
	\maketitle
	\begin{abstract}
		This paper gives two theoretical results on estimating low-rank parameter matrices for linear models with multivariate responses. We first focus on robust parameter estimation of low-rank multi-task learning with heavy-tailed data and quantization scenarios. It comprises two cases: quantization under heavy-tailed responses and quantization with both heavy-tailed covariate and response variables. For each case, our theory shows that the proposed estimator has a minmax near-optimal convergence rate. We then further investigate low-rank linear models with heavy-tailed matrix-type responses. The theory shows that when the random noise has only $(2+\epsilon)$-order moment, our robust estimator still has almost the same statistical convergence rate as that of sub-Gaussian data. Moreover, our simulation experiments confirm the correctness of theories and show the superiority of our estimators.
	\end{abstract}
	\section{Introduction}
	In the era of big data, the computational power of computers is gradually unable to meet the need of efficiently processing massive data. How to bridge the gap between computing power and the huge amount of calculation has become one of the research hotspots in academia and industry. Therefore, in recent years, quantization for high-precision and less important data is more and more popular, which is a signal compression technology that converts floating-point or continuous data into discrete or even integer data. Quantization can effectively compress the model size, reduce the memory space of the data and improve the computational efficiency. For example, in machine learning scenarios, massive data is usually stored in multiple servers far apart. High communication costs and the requirement of user privacy hinder data centralization. Therefore, many literature proposed various divide-and-conquer algorithms, and some of them applied the technique of quantizing stochastic gradients, such as Basu et al. (2019)\cite{Basu}, Seide et al. (2014)\cite{Seide} and  Zheng et al . (2019)\cite{Zheng} et al. Transmitting quantized data between multiple servers can effectively reduce communication costs while maintaining low information loss. On the other hand, in the field of signal processing and statistics, more and more scholars make statistical inference on various statistical models in the context of data quantization. Thrampoulidis (2020)\cite{Thra} considered high-dimensional linear models under uniform quantization and one bit quantization, and demonstrated the near optimality of their estimators. Further, for compressed sensing and matrix completion models with heavy tailed data, Chen et al. (2022)\cite{Chen3} and Chen et al. (2023)\cite{Chen1} designed minimax statistically optimal estimators under one bit quantization and uniform quantization scenarios, respectively. In addition, Chen et al. (2023)\cite{Chen2} studied low-rank linear models with multidresponse under uniformly quantized sub-Gaussian data. Dirksen et al. (2022)\cite{Dirksen2} constructed covariance matrix estimator under one-bit quantization and provided the optimal non asymptotic upper bound when only obtaining two bits for each term of the sample.\par
	For massive data, in addition to the difficulties of training and computing, it is more worthy of attention that it usually has the characteristic of heavy tail. When using the classical statistical method under sub-Gaussian assumption, this phenomenon will lead to high error output and even non convergence of the algorithm. Therefore, some effective approaches to adequately estimate regression parameter with the heavy-tailed noise have been proposed by numerous literature. One of the popular ways is to substitute the traditional square loss with some robust loss functions such as absolute loss, Huber loss (Huber (1973)\cite{Huber}) and Cauchy loss. This type of robust technique was originally aimed at achieving the outlier-robustness. Recently, Fan et al. (2017)\cite{symmetry} first employed the Huber loss into linear regression problem to investigate the robustness against heavy-tailness of the regression error. Their theoretical result unveils that under only finite second order moment condition on the noise, the proposed robust estimator has the same optimal rate as the case of sub-Gaussian tails via carefully tuning the robustification parameter of the Huber loss. It's worth mentioning that another robust loss function proposed by Catoni (2012)\cite{catoni} recently has evoked a growing concern and formed the basis for constructing a series of tail-robust estimators. For example, Minsker (2018) proposed a multivariate extension of Catoni (2012)\cite{catoni}'s estimator and applied the developed estimator to matrix completion with bounded second moment noise. The most eye-catching and convenient methods recently is to shrinkage data (Fan et al. (2021)\cite{shrinkage}). Specifically, Fan et al. (2021)\cite{shrinkage} proposed a shrinkage principle for low-rank matrix recovery with heavy-tailed data. They truncated large heavy-tailed responses or covariates, and used the clipped data into least-squares method. Their theoretical results show that under mild moment constraint condition, the robust estimator achieves nearly the same statistical error rate as the case of sub-Gaussian tails.
	\par
	Due to the above two problems encountered in dealing with big data and the corresponding remedies, in this paper, we study the robust low-rank parameter estimation of linear models with multiresponses under quantization scenarios or heavy-tailed data. Specifically, for the low-rank multi-task learning, we consider both heavy-tailed and quantized case. The difference from directly using the original data for statistical inference is that we first preprocess the data via quantization and shrinkage (if the data is sub-Gaussian, only quantization is required), and then construct and solve the corresponding regularized least squares to obtain the $M$-estimator. For low-rank linear model with matrix response, we only consider the case of heavy-tailed response. Inspired by the robust mean estimator for random matrices proposed by Minsker (2018)\cite{Minsker1}, we introduce a similar robust approach into the heavy-tailed matrix response case. Our theories state that when the random noise of this two models have only finite $(2+\epsilon)$-th moment, the proposed estimators generated by these preprocessed data have the near-optimal convergence rate. In addition, the theory also clarifies that when the covariates and response variables in the multi-task learning model have only finite $4$-th moment, the proposed estimator still possesses the optimality of the statistical error rate. extensive numerical simulations support our theoretical results and show the validity of our robust $M$-estimators.
	\par
	The remainder of this paper is structured as follows. Section \ref{1.1} gives the mathematical notation used in this paper. Section \ref{2} shows the main results of this paper. In Section \ref{3}, simulation experiments are presented. A brief discussion is given in Section \ref{4}. 
	\subsection{Notation}\label{1.1}
	For any positive integer $n$, we denote the set $\{1,2,\ldots,n\}$ by $[n]$. For $a,b \in \mathbb{R}$, $a\vee b$ and $a \wedge b$ stand for the maximum and minimum of $a$ and $b$, respectively. Uppercase letters are denoted as vectors or matrices. For two matrices $X, Y \in \mathbb{R}^{d_{1}\times d_{2}}$, $\langle X,Y\rangle:=\text{tr}(X^{\top}Y)$. The Frobenius norm of $X$ are defined as $\|X\|_{F}=\sqrt{\sum_{i,j}X_{(i,j)}^2}$. The nuclear norm and spectral norm are defined as $\|X\|_{\star}=\text{tr}\left(\sqrt{X^{\top}X}\right)$ and $\|X\|_{\text{op}}=\sqrt{\lambda_{\max}\left(X^{\top}X\right)}$, respectively. For a sub-Gaussian random vector $X\in \mathbb{R}^{d}$, its sub-Gaussian norm is defined as $\|X\|_{\psi_{2}}=\sup_{u\in\mathcal{S}^{d-1}}\left\{\sup_{k\geq 1}\left\{\sqrt[k]{\mathbb{E}[|u^{\top}X|^k]}/\sqrt{k}\right\}\right\}$. We denote $\mathcal{U}\left(\left[-\eta, \eta\right]^{d}\right)$ and $T_{d}(\mu, \Sigma,\nu)$ as the uniform distribution over $\left[-\eta, \eta\right]^{d}$ and $d$-variate $t$ distribution with degree of freedom $\nu$, mean vector $\mu$ and positive definite matrix $\Sigma$, respectively. Given two sequences $\{a_{n}\}_{n=1}^{\infty}$ and $\{b_{n}\}_{n=1}^{\infty}$, we use the notation $a_{n} \asymp b_{n}$, if $b_{n} \lesssim a_{n}\lesssim b_{n}$ where $a_{n}\lesssim b_{n}$ means that there exists a positive constant $C$ such that $a_{n}\leq C b_{n}$ for all $n$. Let $f$ be a real-valued function defined on $\mathbb{R}$ and $A \in \mathbb{R}^{d \times d}$ be a symmetric matrix with the eigenvalue decomposition $A=V \Lambda V^{\top}$. We define $f(A)$ as $f(A)=V f(\Lambda) V^{\top}$, where $f(\Lambda)=$ $\operatorname{diag}\left(f\left(\lambda_1(A)\right), f\left(\lambda_2(A)\right), \ldots, f\left(\lambda_d(A)\right)\right)$.
	\section{Main results}\label{2}
	\subsection{Quantized low-rank multi-task learning with heavy-tailed data}\label{subsec1}
	In this subsection, we consider multi-task learning, 
	i.e., $\left\{(Y_{i},X_{i})\right\}_{i=1}^n$ are i.i.d. sampled from the following model:
	$$Y=\Theta_{\star}^{\top} X+\epsilon$$
	where $X\in \mathbb{R}^{d_{1}}$ and $ Y, \epsilon\in \mathbb{R}^{d_{2}}$ are covariate, response and random noise respectively, and $\mathbb{E}[\epsilon|X]=0_{d_{2}\times 1}$. To estimate low-rank parameter matrix $\Theta_{\star}\in \mathbb{R}^{d_{1}\times d_{2}}$, we consider minimizing the expected $\ell_{2}$ loss which is defined as $\mathbb{E}[\|Y-\Theta^{\top}X\|_{2}^2]=\left\langle\Theta, \Sigma_{XX}\Theta\right\rangle-2\left\langle \Theta, \Sigma_{XY}\right\rangle+C$ where $\Sigma_{XX}:=\mathbb{E}[XX^{\top}]$, $\Sigma_{XY}:=\mathbb{E}[XY^{\top}]$ and $C$ is a constant independent of $\Theta$. Therefore, under different quantization settings, we replace $\Sigma_{XX}$ and $\Sigma_{XY}$ with corresponding estimators $\widetilde{\Sigma}_{XX}$ and $\widetilde{\Sigma}_{XY}$, and construct the empirical loss function $\widehat{\mathcal{L}}_{n}(\Theta)=\left\langle\Theta, \widetilde{\Sigma}_{XX}\Theta\right\rangle-2\left\langle \Theta, \widetilde{\Sigma}_{XY}\right\rangle$. Meanwhile, under low rank structure of $\Theta_{\star}$, we solve the following the regularized least squares program: 
	\begin{equation}\label{eq0}
		\widehat{\Theta}=\underset{\Theta\in \mathbb{R}^{d_{1}\times d_{2}}}{\arg\min}\left\{\widehat{\mathcal{L}}_{n}(\Theta)+\lambda \cdot \left\|\Theta\right\|_{\star}\right\}
	\end{equation}
	where $\lambda$ is a tuning parameter.
	Consider two cases: (a) quantization for heavy-tailed covariates and responses; (b) quantization for only heavy-tailed responses.\par
	Before estimating $\Sigma_{XX}$ and $\Sigma_{XY}$, we need to preprocess the data. i.e., we first truncate the heavy tailed data appropriately, and then quantize the truncated version of data.  Specifically, if $(X_{i}, Y_{i})$ are heavy-tailed, we shrinkage  samples via $\widehat{X}_i:=\frac{\|X_{i}\|_{2}\wedge \tau}{\left\|X_i\right\|_2} X_i$ and $ \widehat{Y}_i:=\frac{\|Y_{i}\|_{2}\wedge \varpi}{\left\|Y_i\right\|_2} Y_i$, where $\tau$ and $\varpi$ are robustification parameters to be specified. If only the response variable is heavy-tailed, we truncate $Y_{i}$. Then  generate random dither 
	$\Lambda_{i1}\sim \mathcal{U}\left(\left[-\frac{\eta_{1}}{2}, \frac{\eta_{1}}{2}\right]^{d_{1}}\right)+\mathcal{U}\left(\left[-\frac{\eta_{1}}{2}, \frac{\eta_{1}}{2}\right]^{d_{1}}\right)$, $\Lambda_{i2}\sim \mathcal{U}\left(\left[-\frac{\eta_{2}}{2}, \frac{\eta_{2}}{2}\right]^{d_{2}}\right)$ and use the quantization function $Q_{\eta}(x):=\eta\left(\left\lfloor\frac{x}{\eta}\right\rfloor+\frac{1}{2}\right), x\in \mathbb{R}$ to quantize each element of the truncated data: 
	$\widetilde{X}_{i}:=Q_{\eta_{1}}\left(\widehat{X}_i+\Lambda_{i1}\right), \widetilde{Y_{i}}:=Q_{\eta_{2}}\left(\widehat{Y}_i+\Lambda_{i2}\right)$. $\eta_{1}$ and $\eta_{2}$ are quantization parameters to quantify the tradeoff between the precision and  discreteness of data.\par
	After truncating and quantizing the data, we construct the corresponding estimators: $\widetilde{\Sigma}_{XX}=\frac{1}{n}\sum_{i=1}^n\widetilde{X}_{i}\widetilde{X}_{i}^{\top}-\frac{\eta_{1}^2}{4}I_{d_{1}}$, $\widetilde{\Sigma}_{XY}=\frac{1}{n}\sum_{i=1}^n\widetilde{X}_{i} \widetilde{Y}_{i}^{\top}$.
	The following lemma presents the statistical error rate of the estimator $\widetilde{\Sigma}_{XX}$ under the spectral norm when $\{X_{i}\}_{i=1}^n$ are heavy-tailed.
	\begin{lemma}\label{lemma1}
		Suppose $\{X_{i}\}_{i=1}^n$ are i.i.d. $d_{1}$-dimensional random vectors with $\sup_{u\in \mathcal{S}^{d_{1}-1}}\mathbb{E}\left[(X_{i}^{\top}u)^4\right]\leq M <\infty$ and $ \mathbb{E}[X_{i}]=0_{d_{1}\times 1}$. By choosing $\tau\asymp\sqrt[4]{\frac{nd_{1}M}{\log(d_{1})}}$, there exists a positive constant $C$ such that for any $ \delta >2,$
		$$\operatorname{P}\left(\left\|\widetilde{\Sigma}_{XX}-\Sigma_{XX}\right\|_{\text{op}}\leq C\delta ( M^\frac{1}{2}+\eta_{1}^2)\sqrt{\frac{ d_{1} \log d_{1}}{n}}\right) \geq 1-2(d_{1}^{1-\delta}+d_{1}^{2-\delta}).$$
	\end{lemma}
	\begin{proof} 
		Let $\xi_{i}:=\widetilde{X}-\widehat{X}_i$, then by Corollary 1 of Chen et al. (2023)\cite{Chen2}, we obtain that for $\forall j \in [d_{1}]$, $|\xi_{ij}|\leq \frac{3}{2}\eta_{1}$.
		Since
		$\widetilde{\Sigma}_{XX}=\frac{1}{n}\sum_{i=1}^n\widetilde{X}_{i}\widetilde{X}^{\top}_{i}-\frac{\eta_{1}^2}{4}I_{d_{1}}=\frac{1}{n}\sum_{i=1}^n(\widehat{X}_i+\xi_{i})(\widehat{X}_i+\xi_{i})^\top-\frac{\eta_{1}^2}{4}I_{d_{1}}=\frac{1}{n}\sum_{i=1}^n\widehat{X}_{i}\widehat{X}_{i}^{\top}+\frac{1}{n}\sum_{i=1}^n\left(\widehat{X}_{i}\xi_{i}^{\top}+\xi_{i}\widehat{X}_{i}^\top\right)+\frac{1}{n}\sum_{i=1}^n\xi_{i}\xi_{i}^{\top}-\frac{\eta_{1}^2}{4}I_{d_{1}}$, we get that
		\begin{equation}\label{eq}
			\left\|\widetilde{\Sigma}_{XX}-\Sigma_{XX}\right\|_{\text{op}}\leq \left\|\widehat{\Sigma}_{n}(\tau)-\Sigma_{XX}\right\|_{\text{op}}\!+\!\left\|\frac{1}{n}\sum_{i=1}^n\left(\widehat{X}_{i}\xi_{i}^{\top}+\xi_{i}\widehat{X}_{i}^\top\right)\right\|_{\text{op}}+\left\|\frac{1}{n}\sum_{i=1}^n\xi_{i}\xi_{i}^{\top}-\frac{\eta_{1}^2}{4}I_{d_{1}}\right\|_{\text{op}}.
		\end{equation}
		For the first term, since
		$\left\|\mathbb{E}\left[\|X_{i}\|_{2}^2X_{i}X_{i}^\top\right]\right\|_{\text{op}}= \sup_{u\in \mathcal{S}^{d_{1}-1}}\sum_{j=1}^{d_{1}}\mathbb{E}\left[X_{ij}^2\left(X_{i}^\top u\right)^2\right]\leq d_{1}M$ and $\tau=\sqrt[4]{\frac{nd_{1}M}{\log(d_{1})}}$, by Lemma 1 of Li et al. (2021)\cite{Li}, we derive that for any $ \delta>2$,
		\begin{equation}\label{eq00}
			\operatorname{P}\left(\left\|\widehat{\Sigma}_n(\tau)-\Sigma_{XX}\right\|_{\mathrm{op}} \leq \delta\sqrt{\frac{M d_{1} \log d_{1}}{n}}\right) \geq 1-2d_{1}^{2-\delta}.    
		\end{equation}
		For the second term, since 
		$\left\|\widehat{X}_{i}\xi_{i}^{\top}\right\|_{\text{op}}=\left\|\widehat{X}_{i}\right\|_{2}\left\|\xi_{i}\right\|_{2}\lesssim \tau \sqrt{d_{1}}\eta_{1}$,
		$$\left\|\mathbb{E}\left[\widehat{X}_{i}\xi_{i}^{\top}\xi_{i}\widehat{X}_{i}^\top\right]\right\|_{\text{op}}=\sup_{u\in \mathbb{S}^{d_{1}-1}}\mathbb{E}\left[\|\xi_{i}\|_{2}^2\left(\widehat{X}_{i}^\top u\right)^2\right]\lesssim d_{1}\eta_{1}^2\sup_{u\in \mathbb{S}^{d_{1}-1}}\mathbb{E}\left[\left(\widehat{X}_{i}^\top u\right)^2\right]\leq d_{1}\eta_{1}^2\sqrt{M},$$
		and 
		$\left\|\mathbb{E}\left[\xi_{i}\widehat{X}_{i}^\top\widehat{X}_{i}\xi_{i}^{\top}\right]\right\|_{\text{op}}=\sup_{u\in \mathbb{S}^{d_{1}-1}}\mathbb{E}\left[\left\|\widehat{X}_{i}\right\|_{2}^2\left(\xi_{i}^\top u\right)^2\right]\lesssim d_{1}\eta_{1}^2\sqrt{M}$, we derive from the matrix Bernstein inequality in Lemma \ref{lemma2} that, for any $\delta>1$,
		\begin{equation}\label{eq2}
			\operatorname{P}\left(\left\|\frac{1}{n}\sum_{i=1}^n\left(\widehat{X}_{i}\xi_{i}^{\top}+\widehat{X}_{i}^\top\xi_{i}\right)\right\|_{\text{op}}\lesssim \eta_{1}\sqrt{\frac{\delta \log(d_{1})d_{1}\sqrt{M}}{n}}\right) \geq 1-d_{1}^{1-\delta}.
		\end{equation}
		As for the third term, because 
		$\|\xi_{i}\xi_{i}^\top\|_{\text{op}}=\left\|\xi_{i}\right\|_{2}^2\leq d_{1} \eta_{1}^2$ and $\left\|\mathbb{E}\left[\left\|\xi_{i}\right\|_{2}^2\xi_{i}\xi_{i}^\top \right]\right\|_{\text{op}}\lesssim d_{1} \eta_{1}^4$, again by Lemma \ref{lemma2}, we then have 
		\begin{equation}\label{eq3}
			\operatorname{P}\left(\left\|\frac{1}{n}\sum_{i=1}^n\xi_{i}\xi_{i}^{\top}-\frac{\eta_{1}^2}{4}I_{d_{1}}\right\|_{\text{op}}\lesssim \eta_{1}^2\sqrt{\frac{\delta d_{1}\log(d_{1})}{n}}\right) \geq 1-d_{1}^{1-\delta}.
		\end{equation}
		Combining (\ref{eq})-(\ref{eq3}), the conclusion can be drawn from the union bound.
	\end{proof}
	Lemma \ref{lemma1} shows that the convergence rate of $\widetilde{\Sigma}_{XX}$ is still optimal after proper data quantization. Based on the above lemma, the following theorem gives the statistical theoretical guarantee for $\widehat{\Theta}$ in (\ref{eq0}).
	\begin{theorem}\label{theorem1}
		Suppose $\text{rank}\left(\Theta_{\star}\right)\leq r$ and there exist two positive constants $\kappa_{0}$ and $R$ such that $\lambda_{\min}(\Sigma_{XX})\geq \kappa_{0}>0$, $\left\|\Theta_{\star}\right\|_{\text{op}}\leq R$.\par
		(a) Further assume that $\sup_{u\in\mathcal{S}^{d_{2}-1}}\mathbb{E}[(Y_{i}^\top u)^4]\leq M < \infty$. Under the condition of the Lemma \ref{lemma1}, for $\forall \delta>2$, by choosing $\varpi \asymp \sqrt[4]{\frac{nd_{2}M}{\log(d_{2})}}$ and $\lambda\asymp \delta(M^\frac{1}{2}+\eta^2)R\sqrt{\frac{d_{\max} \log d_{\max}}{n}}$, there exist positive constants $C_{1}, C_{2}$ only depending on $\kappa_{0}$ such that as long as $n>C_{1}\delta^2(M+\eta_{1}^4)d_{1} \log d_{1}$, we have
		$$
		\operatorname{P}\left(\left\|\widehat{\Theta}-\Theta_{\star}\right\|_{\text{op}}\leq C_{2}\delta\left(M^\frac{1}{2}+\eta^2\right)(R+1)\sqrt{\frac{r\log(d_{\max})d_{\max}}{n}}\right) \leq 1-3d_{\max}^{2-\delta}.
		$$\par
		(b) Suppose that $\exists k>1$ such that $\sup_{u\in\mathcal{S}^{d_{2}-1}}\left(\mathbb{E}[(\mathbb{E}[(\epsilon_{i}^{\top}u)^2|X_{i}])^k]\right)^{1/k}\leq M < \infty$ and $X_{i}$ follows sub-Gaussian distribution with $\mathbb{E}[X_{i}]=0_{d_{1}\times 1}$ and $ \|X_{i}\|_{\psi_{2}}\leq \kappa$. For $\forall \delta>1$, by choosing $\varpi\asymp \sqrt{\frac{n(R+M)}{\log (d_{2})}}$ and $\lambda\asymp \delta(M^\frac{1}{2}+\eta^2)R\sqrt{\frac{d_{\ max} \log d_{\max}}{n}}$, there exist positive constants $c,C_{1}, C_{2}$ only depending on $\kappa_{0},\kappa$ such that as long as $n>C_{1}d_{1}$, we have
		$$
		\operatorname{P}\left(\left\|\widehat{\Theta}-\Theta_{\star}\right\|_{\text{op}}\leq C_{2}\delta\left(M^\frac{1}{2}+\eta^2\right)(R+1)\sqrt{\frac{r\log(d_{\max})d_{\max}}{n}}\right) \leq 1-5d_{\max}^{1-\delta}-4\exp(-cd_{\max})
		$$
		where $d_{\max}:=d_{1}\vee d_{2}$ and $\eta:= \eta_{1}\vee \eta_{2}$.
	\end{theorem}
	\begin{remark}
		Theorem \ref{theorem1} shows that if ignoring the logarithmic factor $\log(d_{\max})$, $\widehat{\Theta}$ in (\ref{eq0}) has almost the same rate of minimax optimal convergence as Fan et al. (2021)\cite{shrinkage} and Chen et al. (2023)\cite{Chen2} under the quantization scenario and heavy-tailed assumption. Note that the quantization parameters $\eta_{1},\eta_{2}$ can be chosen arbitrarily. The larger the parameter level, the more quantize the data but the greater the loss of accuracy. 
	\end{remark}
	\begin{proof}
		(a) By the optimality of $\widehat{\Theta}$, it follows that 
		$\widehat{\mathcal{L}}_{n}(\widehat{\Theta})+\lambda \cdot \|\widehat{\Theta}\|_{\star}\leq \widehat{\mathcal{L}}_{n}(\Theta_{\star})+\lambda \cdot \left\|\Theta_{\star}\right\|_{\star}$. After simple calculations, we show that
		\begin{equation}\label{eq9}
			\begin{aligned}
				\left\langle\widehat{\Theta}-\Theta_\star, \widetilde{\Sigma}_{XX}\left(\widehat{\Theta}-\Theta_\star\right)\right\rangle &\leq 2\left\langle\widetilde{\Sigma}_{X Y}-\widetilde{\Sigma}_{XX}\Theta_\star,\widehat{\Theta}-\Theta_\star\right\rangle +\lambda \cdot \left(\left\|\Theta_\star\right\|_{\star}-\left\|\widehat{\Theta}\right\|_{\star}\right)\\& \leq 2\left\|\widetilde{\Sigma}_{X Y}-\widetilde{\Sigma}_{XX}\Theta_\star\right\|_{\text{op}}\left\|\widehat{\Theta}-\Theta_\star\right\|_{\star}+\lambda\cdot \left\|\Theta_\star-\widehat{\Theta}\right\|_{\star}.
			\end{aligned}
		\end{equation}
		From the condition $\lambda_{\min}(\Sigma_{XX})>\kappa_{0}$ and Lemma \ref{lemma1}, we obtain that when $n\gtrsim \delta^2d_{1}\log(d_{1}) (M+\eta_{1}^4)$,  $\lambda_{\min}\left(\widetilde{\Sigma}_{XX}\right)\geq \kappa_{0}/2$. Therefore, 
		\begin{equation}\label{eq10}
			\begin{aligned}
				\left\langle\widehat{\Theta}-\Theta_\star, \widetilde{\Sigma}_{XX}\left(\widehat{\Theta}-\Theta_\star\right)\right\rangle&= \sum_{i=1}^{d_{2}}\left(\widehat{\theta}^{(i)}-\theta^{(i)}_{\star}\right)\widetilde{\Sigma}_{XX}\left(\widehat{\theta}^{(i)}-\theta^{(i)}_{\star}\right)^{\top}\\& \geq \lambda_{\min}(\widetilde{\Sigma}_{XX})\sum_{i=1}^{d_{2}}\left\|\widehat{\theta}^{(i)}-\theta^{(i)}_{\star}\right\|_{2}^2
				\geq \frac{\kappa_{0}}{2} \left\|\widehat{\Theta}-\Theta_\star\right\|_{F}^2.
			\end{aligned}
		\end{equation}
		On the other hand, 
		$\left\|\widetilde{\Sigma}_{X Y}-\widetilde{\Sigma}_{XX}\Theta_\star\right\|_{\text{op}}
		\leq \underbrace{\left\|\widetilde{\Sigma}_{X Y}- \Sigma_{X Y}\right\|_{\text{op}}}_{T_{1}}+\underbrace{\left\|\left(\widetilde{\Sigma}_{XX}-\Sigma_{XX}\right)\Theta_\star\right\|_{\text{op}}}_{T_{2}}.$
		For 
		$T_{1}$, denoting that 
		$\zeta_{i}:=\widetilde{Y_{i}}-\widehat{Y}_i$, from 
		Chen et al. (2023)\cite{Chen2}, we have the fact that
		$\forall j \in [d_{2}]$, $|\zeta_{ij}|\leq \frac{3}{2}\eta_{2}$ and 
		$\mathbb{E}\left[\xi_{i}\zeta_{i}^{\top}\right]=0_{d_{1}\times d_{2}}$. Besides, 
		$\widetilde{\Sigma}_{XY}=\frac{1}{n}\sum_{i=1}^n\widetilde{X}_{i}\widetilde{Y_{i}}^\top =\frac{1}{n}\sum_{i=1}^n(\widehat{X}_i+\xi_{i})(\widehat{Y}_i+\zeta_{i})^\top=\frac{1}{n}\sum_{i=1}^n\widehat{X}_{i}\widehat{Y}_{i}^{\top}+\frac{1}{n}\sum_{i=1}^n\left(\widehat{X}_{i}\zeta_{i}^{\top}+\xi_{i}\widehat{Y}_{i}^\top\right)+\frac{1}{n}\sum_{i=1}^n\xi_{i}\zeta_{i}^{\top}$. Therefore,
		\begin{equation}\label{eq11}
			\left\|\widetilde{\Sigma}_{XY}\!-\!\Sigma_{XY}\right\|_{\text{op}}\leq \left\|\frac{1}{n}\sum_{i=1}^n\widehat{X}_{i}\widehat{Y}_{i}^{\top}\!-\!\Sigma_{XY}\right\|_{\text{op}}\!+\!\left\|\frac{1}{n}\sum_{i=1}^n\left(\widehat{X}_{i}\zeta_{i}^{\top}+\xi_{i}\widehat{Y}_{i}^\top\right)\right\|_{\text{op}}\!+\!\left\|\frac{1}{n}\sum_{i=1}^n\xi_{i}\zeta_{i}^{\top}\right\|_{\text{op}}.
		\end{equation}
		For the first term on the right-hand side of (\ref{eq11}),
		$$\left\|\frac{1}{n}\sum_{i=1}^n\widehat{X}_{i}\widehat{Y}_{i}^{\top}-\Sigma_{XY}\right\|_{\text{op}}\leq 
		\left\|\frac{1}{n}\sum_{i=1}^n\widehat{X}_{i}\widehat{Y}_{i}^{\top}-\mathbb{E}\left[\widehat{X}_{i}\widehat{Y}_{i}^{\top}\right]\right\|_{\text{op}}+\left\|\mathbb{E}\left[\widehat{X}_{i}\widehat{Y}_{i}^{\top}\right]-\Sigma_{XY}\right\|_{\text{op}}.$$
		For the first term on the right-hand side in the above inequality, we can bound it by Lemma \ref{lemma2}. 
		Let $S_{i}:=\widehat{X}_{i}\widehat{Y}_{i}^{\top}-\mathbb{E}\left[\widehat{X}_{i}\widehat{Y}_{i}^{\top}\right]$. On the one hand,
		$$
		\begin{aligned}
			\left\|S_{i}\right\|_{\text{op}}&\leq \left\|\widehat{X}_{i}\widehat{Y}_{i}^{\top}\right\|_{\text{op}}+\left\|\mathbb{E}[\widehat{X}_{i}\widehat{Y}_{i}^{\top}]\right\|_{\text{op}}=\left\|\widehat{X}_{i}\right\|_{2}\left\|\widehat{Y}_{i}\right\|_{2}+\sup_{\substack{u\in \mathcal{S}^{d_{1}-1}\\v\in \mathcal{S}^{d_{2}-1}}}\mathbb{E}[u^\top\widehat{X}_{i}\widehat{Y}_{i}^{\top}v]\\&\leq \tau \varpi+\sup_{\substack{u\in \mathcal{S}^{d_{1}-1}\\v\in \mathcal{S}^{d_{2}-1}}}\sqrt{\mathbb{E}[(u^\top{X}_{i})^2]\mathbb{E}[({Y}_{i}^{\top}v)^2]}\leq \tau\varpi +\sqrt{M}.\end{aligned}
		$$
		On the other hand, since
		$$\left\|\mathbb{E}[\widehat{X}_{i}\widehat{Y}_{i}^{\top}\widehat{Y}_{i}\widehat{X}_{i}^\top]\right\|_{\text{op}}\leq \sup_{u\in \mathcal{S}^{d_{1}-1}}\sum_{j=1}^{d_{1}}\mathbb{E}\left[Y_{ij}^2(\widehat{X}_{i}^\top u)^2\right]\leq \sup_{u\in \mathcal{S}^{d_{1}-1}}\sum_{j=1}^{d_{1}}\sqrt{\mathbb{E}Y_{ij}^4\mathbb{E}[(X_{i}^\top u)^4]}\leq d_{1}M,$$
		$\left\|\mathbb{E}[\widehat{Y}_{i}\widehat{X}_{i}^\top\widehat{X}_{i}\widehat{Y}_{i}^{\top}]\right\|_{\text{op}}\leq d_{2}M$ and
		$\left\|\mathbb{E}[\widehat{X}_{i}\widehat{Y}_{i}^{\top}]\mathbb{E}[\widehat{Y}_{i}\widehat{X}_{i}^\top]\right\|_{\text{op}}\vee \left\|\mathbb{E}[\widehat{Y}_{i}\widehat{X}_{i}^\top]\mathbb{E}[\widehat{X}_{i}\widehat{Y}_{i}^{\top}]\right\|_{\text{op}}\leq \left\|\mathbb{E}[\widehat{X}_{i}\widehat{Y}_{i}^{\top}]\right\|_{\text{op}}^2 \leq M$,
		thus
		$\left\|\mathbb{E}[S_{i}S_{i}^\top]\right\|_{\text{op}}\vee \left\|\mathbb{E}[S_{i}^\top S_{i}]\right\|_{\text{op}}\leq (d_{\max}+1)M.$ It follows from the matrix Bernstein inequality that 
		$$
		\operatorname{P}\left(\left\|\frac{1}{n}\sum_{i=1}^n\widehat{X}_{i}\widehat{Y}_{i}^{\top}-\mathbb{E}\left[\widehat{X}_{i}\widehat{Y}_{i}^{\top}\right]\right\|_{\text{op}}>t\right)\leq  (d_{1}+d_{2}) \exp \left(\frac{-n t^2/2}{(d_{\max}+1)M+\left(\tau\varpi +\sqrt{M}\right) t/3}\right).
		$$
		Let $t=\sqrt{\frac{\delta Md_{\max}\log(d_{\max})}{n}}$, then
		\begin{equation}\label{eq12}
			\operatorname{P}\left(\left\|\frac{1}{n}\sum_{i=1}^n\widehat{X}_{i}\widehat{Y}_{i}^{\top}-\mathbb{E}\left[\widehat{X}_{i}\widehat{Y}_{i}^{\top}\right]\right\|_{\text{op}}\lesssim\sqrt{\frac{\delta Md_{\max}\log(d_{\max})}{n}}\right)\geq  1-d_{\max}^{1-\delta}.
		\end{equation}
		Since for 
		$\forall u\in \mathcal{S}^{d_{1}-1}$ and $ v\in \mathcal{S}^{d_{2}-1}$, 
		\begin{equation}\label{eq13}
			\begin{aligned}
				&\mathbb{E}\left[u^{\top}\left(\widehat{X}_{i}\widehat{Y}_{i}^{\top}-X_{i}Y_{i}^{\top}\right)v\right]\leq
				\mathbb{E}\left[|u^{\top}X_{i}Y_{i}^{\top}v|1_{\left\{\|X_{i}\|_{2}\geq \tau \text{ or } \|Y_{i}\|_{2}\geq \varpi\right\}}\right]\\&\leq \sqrt{\mathbb{E}\left[(u^{\top}X_{i})^2(Y_{i}^{\top}v)^2\right]\operatorname{P}(\{\|X_{i}\|_{2}\geq \tau\}\cup\{\|Y_{i}\|_{2}\geq \varpi\})}\leq \sqrt{M}\sqrt{\frac{\mathbb{E}\|X_{i}\|_{2}^4}{\tau^4}+\frac{\mathbb{E}\|Y_{i}\|_{2}^4}{\varpi^4}}\\&
				\leq \sqrt{M}\left(d_{1}\sqrt{M}/\tau^2\!+\!d_{2}\sqrt{M}/\varpi^2\right)\lesssim \sqrt{\frac{M}{n}}\left(\sqrt{d_{1}\log(d_{1})}\!+\!\sqrt{d_{2}\log(d_{2})}\right)\lesssim \sqrt{\frac{Md_{\max}\log(d_{\max})}{n}},
			\end{aligned}
		\end{equation}
		where the fourth inequality follows from $C_{r}$ inequality, it shows that 
		$\left\|\mathbb{E}\left[\widehat{X}_{i}\widehat{Y}_{i}^{\top}\right]-\Sigma_{XY}\right\|_{\text{op}}\lesssim \sqrt{\frac{Md_{\max}\log(d_{\max})}{n}}$. The second and third terms of (\ref{eq11}) can be bounded via (\ref{eq2}) and (\ref{eq3}) that 
		\begin{equation}\label{eq4}
			\operatorname{P}\left(\left\|\frac{1}{n}\sum_{i=1}^n\left(\widehat{X}_{i}\zeta_{i}^{\top}+\xi_{i}\widehat{Y}_{i}^{\top}\right)\right\|_{\text{op}}\lesssim (\eta_{1}+\eta_{2})\sqrt{\frac{\delta \log(d_{\max})d_{\max}\sqrt{M}}{n}}\right) \geq 1-d_{\max}^{1-\delta},
		\end{equation}
		\begin{equation}\label{eq6}
			\operatorname{P}\left(\left\|\frac{1}{n}\sum_{i=1}^n\xi_{i}\zeta_{i}^{\top}\right\|_{\text{op}}\lesssim \eta_{1}\eta_{2}\sqrt{\frac{\delta d_{\max}\log(d_{\max})}{n}}\right) \geq 1-d_{\max}^{1-\delta}.
		\end{equation}
		For $T_{2}$, we have that $\left\|\left(\widetilde{\Sigma}_{XX}-\Sigma_{XX}\right)\Theta_\star\right\|_{\text{op}}\leq R\left\|\widetilde{\Sigma}_{XX}-\Sigma_{XX}\right\|_{\text{op}}\leq \delta (M^\frac{1}{2}+\eta_{1}^2)R\sqrt{\frac{ d_{1} \log d_{1}}{n}}$. By the proof of Theorem 1 in Fan et al. (2021)\cite{shrinkage}, it follows that 
		$\left\|\Theta_\star-\widehat{\Theta}\right\|_{\star}\leq \sqrt{r}\left\|\Theta_\star-\widehat{\Theta}\right\|_{\text{op}}$. Combining (\ref{eq9})-(\ref{eq6}) and choosing $\lambda\asymp \delta(M^\frac{1}{2}+\eta^2)R\sqrt{\frac{d_{\max} \log d_{\max}}{n}}$ yields that with probability at least $1-3d_{\max}^{1-\delta}$, 
		$$
		\begin{aligned}
			\frac{\kappa_{0}}{2} \left\|\widehat{\Theta}-\Theta_\star\right\|_{F}^2&\lesssim \delta \left(M^\frac{1}{2}+\eta^2\right)(R+1)\sqrt{\frac{d_{\max} \log d_{\max}}{n}}\left\|\widehat{\Theta}-\Theta_\star\right\|_{\star}+\lambda \left\|\widehat{\Theta}-\Theta_\star\right\|_{\star}\\
			&\lesssim \delta \left(M^\frac{1}{2}+\eta^2\right)(R+1)\sqrt{rd_{\max} \log d_{\max}/n}\left\|\widehat{\Theta}-\Theta_\star\right\|_{F}.
		\end{aligned}
		$$
		(b) The proof is similar with that of (a), except for the upper bound of $\left\|\widetilde{\Sigma}_{XX}-\Sigma_{XX}\right\|_{\text{op}} $and $\left\|\widetilde{\Sigma}_{XY}-\Sigma_{XY}\right\|_{\text{op}}$. Since $X_{i}$ follows sub-Gaussian distribution, by Lemma 4 of Chen et al. (2023)\cite{Chen2}, it follows that there exists a positive constant $c$ such that for $\forall \delta>0$
		$$\operatorname{P}\left(\left\|\widetilde{\Sigma}_{XX}-\Sigma_{XX}\right\|_{\text{op}} \lesssim \frac{(\kappa+c\eta_{1})^2}{\kappa_{0}+\eta_{1}^2/4}\left(\sqrt{\frac{d_1+\delta}{n}}+\frac{d_1+\delta}{n}\right)\right)\geq 1-2\exp(-\delta).
		$$
		On the other hand, letting 
		$\xi_{i}=\widetilde{X}_{i}-X_{i}$, since
		$$\left\|\widetilde{\Sigma}_{XY}-\Sigma_{XY}\right\|_{\text{op}}\leq \left\|\frac{1}{n}\sum_{i=1}^n X_{i}\widehat{Y}_{i}^{\top}-\Sigma_{XY}\right\|_{\text{op}}+\left\|\frac{1}{n}\sum_{i=1}^n{X}_{i}\zeta_{i}^{\top}\right\|_{\text{op}}+\left\|\frac{1}{n}\sum_{i=1}^n\xi_{i}\widehat{Y}_{i}^{\top}\right\|_{\text{op}}+\left\|\frac{1}{n}\sum_{i=1}^n\xi_{i}\zeta_{i}^{\top}\right\|_{\text{op}},$$
		the first and second terms on the right-hand side can be obtained from Lemma 5 of Fan et al. (2021)\cite{shrinkage} and Lemma 3 of Chen et al. (2023)\cite{Chen2}, respectively. i.e.,
		$$
		\operatorname{P}\left(\left\|\frac{1}{n}\sum_{i=1}^n X_{i}\widehat{Y}_{i}^{\top}-\Sigma_{XY}\right\|_{\mathrm{op}} \lesssim \delta\sqrt{\frac{\left(R+M\right)d_{\max} \log \left(d_{\max}\right)}{n}}\right) \geq 1-3d_{\max}^{1-\delta},
		$$
		\begin{equation}\label{eq7}
			\operatorname{P}\left(\left\|\frac{1}{n}\sum_{i=1}^n X_{i}\zeta_{i}^{\top}\right\|_{\text{op}}\lesssim \kappa\eta_{2}\sqrt{\frac{d_{\max}}{n}}\right) \geq 1-2\exp \left(-cd_{\max}\right).
		\end{equation}
		Finally, the upper bounds on the last two terms are the same as (\ref{eq4}) and (\ref{eq6}).
	\end{proof}
	\subsection{Low-rank linear regression model for heavy-tailed matrix responses}
	Next, we study low-rank linear regression with heavy-tailed matrix-type responses. 
	\begin{equation}\label{eq1}
		Y=\sum_{k=1}^s x_{(k)} \Theta_\star^{(k)}+E,
	\end{equation}
	where $\left\{\Theta_\star^{(k)}\in \mathbb{R}^{d_{1}\times d_{2}}\right\}_{k\in [s]}$ is $s$ parameter matrices to be estimated, $X=\left(x_{(1)}, x_{(2)}, \cdots ,x_{(s)} \right)^{\top}$ is $s$-dimensional covariate. $Y, E \in \mathbb{R}^{d_{1}\times d_{2}}$ are the response matrix and the random noise matrix, respectively and $\mathbb{E}[E|X]=0_{d_{1}\times d_{2}}$. When $\min\left\{s, d_{1}, d_{2}\right\}$ is relatively large, in order to efficiently estimate the parameters $\Theta_{\star}:=\left(\Theta_\star^{(1)}, \Theta_\star^{(2)},\cdots, \Theta_\star^ {(s)}\right)$, we need to make some structural assumptions on the parameter matrix: (1) $\sum_{k=1}^s \text{rank}\left(\Theta_\star^{(k)}\right)\leq r\ll s\min\{d_{1},d_{2}\}$;  (2) $\max_{k\in [s]} \left\|\Theta_\star^{(k)}\right\|_{\text{op}}\leq R=O(1)$ where $R$ is a positive constant. Assumption (1) requires that for $\forall k \in [s],$  $\Theta_{\star}^{(k)}$ is of low rank. Assumption (2) requires that $\left\|\Theta_{\star}^{(k)}\right\|_{\text{op}}$ should be relatively small, and a similar condition can be found in Chen et al (2023)\cite{Chen2}.\par
	Kong et al (2020)\cite{Kong}, Hao et al (2021)\cite{Hao} and Chen et al (2023)\cite{Chen2} have studied such multivariate regression models with matrix-type responses. The purpose of this work is to deal with the heavy-tailed data, which often occurs in the big data and high-dimensional data. We transform the expected $\ell_{2}$ loss into the following expression.
	$$
	\begin{aligned}
		&\mathbb{E}\ell(\Theta)=\mathbb{E}\operatorname{tr}\left[\left( Y-\sum_{k=1}^s x_{(k)} \Theta^{(k)}\right)\left(Y-\sum_{k=1}^s x_{(k)} \Theta^{(k)}\right)^{\top}\right]\\=&\mathbb{E}\operatorname{tr}\left[YY^{\top}\right]+
		\mathbb{E}\operatorname{tr}\left[\left(\sum_{k=1}^s x_{(k)} \Theta^{(k)}\right)\left(\sum_{k=1}^s x_{(k)} \Theta^{(k)}\right)^{\top}\right]-2\sum_{k=1}^s\mathbb{E}\left[x_{(k)}\left\langle Y, \Theta^{(k)}\right\rangle\right]\\=& \sum_{i,j\in [s]}\mathbb{E}\left[x_{(i)}x_{(j)}\right]\left\langle\Theta^{(i)},\Theta^{(j)}\right\rangle -2\sum_{k=1}^s\left\langle \mathbb{E}\left[x_{(k)}Y\right], \Theta^{(k)}\right\rangle=:\left\langle \Sigma_{XX},  \Pi \right\rangle -2\sum_{k=1}^s\left\langle\Sigma_{x_{(k)}Y}, \Theta^{(k)}\right\rangle,
	\end{aligned}
	$$
	where $\Sigma_{XX}:=\mathbb{E}\left[XX^{\top}\right], \Sigma_{x_{(k)}Y}:=\mathbb{E}[x_{(k)}Y]$ and $(\Pi)_{i,j}:=\left\langle\Theta^{(i)},\Theta^{(j)}\right\rangle$. We omit $\mathbb{E}\operatorname{tr}\left[YY^{\top}\right]$ which is not related to $\Theta:=\left(\Theta^{(1)}, \Theta^{(2)},\cdots, \Theta^{(s)}\right)$, because it does not affect the evaluation of the parameters. \par
	To reduce the impact of heavy-tailed data on parameter estimation, we replace $\Sigma_{x_{(k)}Y}$ with some
	tailed-robust matrix estimator. Specifically, let $\{(X_{i}, Y_{i})\}_{i=1}^n$ be $n$ i.i.d. samples from (\ref{eq1}), the covariate $X_{i}$ follows the sub-Gaussian distribution and the random error matrix $E_{i}$ follows the heavy-tailed distribution, then let 
	$\widehat{H}_{11}^{(k)} \in \mathbb{R}^{d_1 \times d_1}, \widehat{H}_{22}^{(k)} \in \mathbb{R}^{d_2 \times d_2}, \widehat{\Sigma}_{x_{(k)}Y} \in \mathbb{R}^{d_1 \times d_2}$ such that 
	$$\left(\begin{array}{cc}
		\widehat{H}_{11}^{(k)} & \widehat{\Sigma}_{x_{(k)}Y} \\
		\widehat{\Sigma}_{x_{(k)}Y}^{\top} & \widehat{H}_{22}^{(k)}
	\end{array}\right)=\frac{1}{n}\sum_{i=1}^n\psi_{\tau_{k}}\left(\mathcal{F}\left(x_{i(k)}Y_{i}\right)\right), \text{where}\quad 
	\mathcal{F}\left(x_{i(k)}Y_{i}\right):=\left(\begin{array}{cc}
		0_{d_{1}\times d_{1}} & x_{i(k)}Y_{i} \\
		x_{i(k)}Y_{i}^{\top} & 0_{d_{2}\times d_{2}}
	\end{array}\right).
	$$
	where $\psi_{\tau}(x):=\text{sign}(x)\cdot(|x|\wedge \tau)$. We choose $\widehat{\Sigma}_{x_{(k)}Y}$ as the estimator of $\Sigma_{x_{(k)}Y}$, where $\{\tau_{k}\}_{k\in [s]}$ is pre-determined thresholds to balance the tail-robustness with the bias of the estimation. This robust technology was first proposed by Minsker (2018)\cite{Minsker1}. On the other hand, we use the sample covariance matrix $\widehat{\Sigma}_{XX}:=\frac{1}{n}\sum_{i \in [n]} X_{i}X_{i}^{\top}$ as the estimator of $\Sigma_{XX}$. Therefore, we define the robust empirical $\ell_{2}$ loss as $\widehat{\ell}_{n}(\Theta):=\left\langle \widehat{\Sigma}_{XX}, \Pi \right\rangle -2\sum_{k=1}^s\left\langle\widehat{\Sigma}_{x_{(k)}Y}, \Theta^{(k)}\right\rangle$ and solve the following optimization problem to obtain the $M$-estimate of $\Theta_{\star}$.
	\begin{equation}\label{eq01}
		\widehat{\Theta}=\underset{\Theta\in \mathbb{R}^{d_{1}\times sd_{2}}}{\arg\min}\left\{\widehat{\ell}_{n}(\Theta)+\lambda \cdot \sum_{i=1}^s\left\|\Theta^{(i)}\right\|_{\star}\right\},
	\end{equation}
	where the penalty term $\lambda \cdot\sum_{i=1}^s\left\|\Theta^{(i)}\right\|_{\star}$ is added to recover the low-rank parameter $\Theta_{\star}$.\par
	The following theorem gives the theoretical guarantee for $\widehat{\Theta}$.
	\begin{theorem}\label{theorem2}
		Suppose $s\asymp \sqrt{d_{1}+d_{2}}$ and the parameter matrix $\Theta_{\star}$ satisfies the above structural conditions. $X_{i}$ follows the sub-Gaussian distribution with $\|X_{i}\|_{\psi_{2}}\leq \kappa$ and $\lambda_{\min}(\Sigma_{XX})\geq \kappa_{0}>0$. If $\exists \ell >1$ such that $\sup_{\substack{u\in \mathcal{S}^{d_{1}-1}\\v\in \mathcal{S}^{d_{2}-1}}}\sqrt[\ell]{\mathbb{E}\left(\mathbb{E}\left[(u^{\top}E_{i}v)^2|X_{i}\right]\right)^\ell}\leq M<\infty$, by choosing $\tau_{k}\asymp\sigma_{k}\sqrt{n/\log(d_{1}+d_{2})}$ and $\lambda \asymp R\delta \sqrt{\frac{M\left(d_1+d_ 2\right) \log \left(d_1+d_2\right)}{n}}$ for $k \in [s]$, there exist positive constants $C_{1}, C_{2}$ only depending on $\kappa, \kappa_{0}$ such that as long as $n> C_{1} (s+\delta \log(s))$, we have for $\forall \delta>\frac{5}{2}$,
		$$\operatorname{P}\left(\left\|\widehat{\Theta}-\Theta_{\star}\right\|_{F}\leq C_{2} R\delta \sqrt{\frac{Mr\left(d_1+d_2\right) \log \left(d_1+d_2\right)}{n}}\right)\geq 1-2(d_{1}+d_{2})^{5/2-\delta}-2(d_{1}+d_{2})^{-\delta/2}$$
		where $\sigma_{k}^2:= \left\|\mathbb{E}\left[x_{i(k)}^2Y_{i}Y_{i}^{\top}\right]\right\|_{\text{op}}\vee \left\|\mathbb{E}\left[x_{i(k)}^2Y_{i}^{\top}Y_{i}\right]\right\|_{\text{op}}$.
	\end{theorem}
	\begin{remark}
		Theorem \ref{theorem2} states that if ignoring the logarithmic factor $\log(d_{1}+d_{2})$, $\widehat{\Theta}$ in (\ref{eq01}) has the minimax optimal convergence rate under the $(2+\epsilon)$-order moment random noise assumption.
	\end{remark}
	\begin{remark}
		An interesting idea is that if the covariates are also heavy-tailed, whether it is possible to employ the tail robust covariance estimator of $\Sigma_{XX}$ in subsection \ref{subsec1} to address the parameter estimation problem for heavy tailed covariates and response variables? Through extensive experimental attempts, we found that this estimator is not superior to traditional least squares estimate although it possesses theoretical  feasibility. Therefore, constructing an effective robust estimator is a topic worth further consideration.
	\end{remark}
	\begin{proof}
		By the optimality of $\widehat{\Theta}$, it follows that $\widehat{\ell}_{n}(\widehat{\Theta})+\lambda \cdot \sum_{i=1}^s\left\|\widehat{\Theta}^{(i)}\right\|_{\star}\leq \widehat{\ell}_{n}(\Theta_{\star})+\lambda \cdot \sum_{i=1}^s\left\|\Theta^{(i)}_{\star}\right\|_{\star}$. By simple calculations, we derive that
		\begin{equation}\label{eq14}
			\begin{gathered}
				\left\langle\widetilde{\Delta} \widetilde{\Delta}^{\top}, \widehat{\Sigma}_{XX}\right\rangle \leq 2\sum_{k=1}^s\left\langle\widehat{\Sigma}_{x_{(k)} Y},\widehat{\Theta}^{(k)}-\Theta_\star^{(k)}\right\rangle -2\sum_{i=1}^s\left\langle\sum_{j=1}^s\left(\widehat{\Sigma}_{XX}\right)_{i,j} \Theta_\star^{(j)}, \widehat{\Theta}^{(i)}-\Theta_\star^{(i)}\right\rangle \\
				+\lambda \cdot \sum_{i=1}^s\left(\left\|\Theta_\star^{(i)}\right\|_{\star}-\left\|\widehat{\Theta}^{(i)}\right\|_{\star}\right)
			\end{gathered} 
		\end{equation}
		where $\widetilde{\Delta}:=\left[\text{vec}\left(\widehat{\Theta}^{(1)}-\Theta_{\star}^{(1)}\right), \cdots, \text{vec}\left(\widehat{\Theta}^{(s)}-\Theta_{\star}^{(s)}\right)\right]^{\top}$. Since $\left\|X_{i}\right\|_{\psi_{2}}\leq \kappa$, by Exercise 4.7.3 in Vershynin (2018)\cite{Vershynin}, we have
		$$
		\operatorname{P}\left(\left\|\widehat{\Sigma}_{XX}-\Sigma_{XX}\right\|_{\mathrm{op}} \leq C\kappa^2\left(\sqrt{\frac{s+\delta\log(s)}{n}}+\frac{s+\delta\log(s)}{n}\right)\left\|\Sigma_{XX}\right\|_{\text{op}}\right)\geq 1-2s^{-\delta}.
		$$
		As long as $n\gtrsim s+\delta \log(s)$, we have $\lambda_{\min}\left(\widehat{\Sigma}_{XX}\right)\geq\frac{1}{2}\lambda_{\min}(\Sigma_{XX})>\frac{\kappa_0}{2}$. Therefore,
		$\left\langle\widetilde{\Delta} \widetilde{\Delta}^{\top}, \widehat{\Sigma}_{XX}\right\rangle \geq \frac{\kappa_0}{2}\|\widetilde{\Delta}\|_F^2=\frac{\kappa_0}{2}\|\Delta\|_F^2$. 
		Since $\Sigma_{x_{(k)} Y}=\sum_{j=1}^s \left(\Sigma_{XX}\right)_{k,j}\Theta_{\star}^{(j)}$, it follows that
		\begin{equation}\label{eq15}
			\begin{aligned}
				&\left|\sum_{k=1}^s\left\langle\widehat{\Sigma}_{x_{(k)} Y},\widehat{\Theta}^{(k)}-\Theta_\star^{(k)}\right\rangle -\sum_{k=1}^s\left\langle\sum_{j=1}^s\left(\widehat{\Sigma}_{XX}\right)_{k,j} \Theta_\star^{(j)}, \widehat{\Theta}^{(k)}-\Theta_\star^{(k)}\right\rangle\right|\\\leq& \left|\sum_{k=1}^s\left\langle\widehat{\Sigma}_{x_{(k)} Y}-\Sigma_{x_{(k)}Y},\widehat{\Theta}^{(k)}-\Theta_\star^{(k)}\right\rangle\right|+\left|\sum_{k=1}^s\left\langle\sum_{j=1}^s\left(\left(\widehat{\Sigma}_{XX}\right)_{k,j}-\left(\Sigma_{XX}\right)_{k,j}\right) \Theta_\star^{(j)}, \widehat{\Theta}^{(k)}-\Theta_\star^{(k)}\right\rangle\right| \\\leq& \left(\max _{k \in[s]}\underbrace{\left\|\widehat{\Sigma}_{x_{(k)} Y}-\Sigma_{x_{(k)}Y}\right\|_{\text{op}}}_{T_{1}}+\max _{k \in[s]}\underbrace{\left\|\sum_{j=1}^s\left(\widehat{\Sigma}_{XX}-\Sigma_{XX}\right)_{k,j} \Theta_\star^{(j)}\right\|_{\text{op}}}_{T_2}\right)\left(\sum_{k=1}^s\left\|\widehat{\Theta}^{(k)}-\Theta_\star^{(k)}\right\|_{\star}\right).
			\end{aligned} 
		\end{equation}
		For $T_{1}$, 
		$$
		\begin{gathered}
			\left\|\mathbb{E}\left[x_{i(k)}^2Y_{i}Y_{i}^{\top}\right]\right\|_{\text{op}}=\left\|\mathbb{E}\left[x_{i(k)}^2\left(\sum_{k=1}^sx_{i(k)}\Theta_{\star}^{(k)}+E_{i}\right)\left(\sum_{k=1}^sx_{i(k)}\Theta_{\star}^{(k)}+E_{i}\right)^{\top}\right]\right\|_{\text{op}}\leq 
			\underbrace{\left\|\mathbb{E}\left[x_{i(k)}^2E_{i}E_{i}^{\top}\right]\right\|_{\text{op}}}_{I_{1}}\\+\underbrace{\left\|\mathbb{E}\left[x_{i(k)}^2\left(\sum_{k=1}^sx_{i(k)}\Theta_{\star}^{(k)}\right)\left(\sum_{k=1}^sx_{i(k)}\Theta_{\star}^{(k)}\right)^{\top}\right]\right\|_{\text{op}}}_{I_{2}}+2\underbrace{\left\|\mathbb{E}\left[x_{i(k)}^2\left(\sum_{k=1}^sx_{i(k)}\Theta_{\star}^{(k)}\right)E_{i}^{\top}]\right]\right\|_{\text{op}}}_{I_{3}}.
		\end{gathered}
		$$
		For $I_{1}$, we have that
		$$
		\begin{aligned}
			I_{1}&=\sup_{u\in \mathcal{S}^{d_{1}-1}}\mathbb{E}\left[x_{i(k)}^2u^{\top}E_{i}E_{i}^{\top}u\right]=\sup_{u\in \mathcal{S}^{d_{1}-1},v\in \mathcal{S}^{d_{2}-1}}\mathbb{E}\left[x_{i(k)}^2(u^{\top}E_{i}v)^2\right]\\&\leq \left(\mathbb{E}|x_{i(k)}|^{\frac{2\ell}{\ell-1}}\right)^\frac{\ell-1}{\ell}\sup_{u\in \mathcal{S}^{d_{1}-1},v\in \mathcal{S}^{d_{2}-1}}\sqrt[\ell]{\mathbb{E}\left(\mathbb{E}\left[(u^{\top}E_{i}v)^2|X_{i}\right]\right)^\ell}
			\leq \left(\kappa_0^2 \frac{2\ell}{\ell-1}\right) M.
		\end{aligned}
		$$
		For $I_{2}$, since $\mathbb{E}\left[x_{i(k)}^2x_{i(l)}x_{i(m)}\right]\leq \sqrt{\mathbb{E}\left(x_{i(k)}^4\right)\sqrt{\mathbb{E}\left(x_{i(l)}^4\right)\mathbb{E}\left(x_{i(m)}^4\right)}}\leq 16\kappa_{0}^4$, we obtain that
		$$
		\begin{aligned}
			I_{2}&=\left\|\sum_{l=1}^s\sum_{m=1}^s\mathbb{E}\left[x_{i(k)}^2x_{i(l)}x_{i(m)}\right]\Theta_{\star}^{(l)}\Theta_{\star}^{(m)\top}\right\|_{\text{op}}\leq 16\kappa_{0}^4\sum_{l=1}^s\sum_{m=1}^s\left\|\Theta_{\star}^{(l)}\Theta_{\star}^{(m)\top}\right\|_{\text{op}}\\&\leq 16\kappa_{0}^4\sum_{l=1}^s\sum_{m=1}^s\left\|\Theta_{\star}^{(l)}\right\|_{\text{op}}\left\|\Theta_{\star}^{(m)}\right\|_{\text{op}}\leq 16s^2\kappa_{0}^4R^2.
		\end{aligned}
		$$
		For $I_{3}$, we derive that
		$$
		\begin{aligned}
			I_{3}&\leq \sum_{l=1}^s\left\|\mathbb{E}\left[x_{i(k)}^2x_{i(l)}\Theta_{\star}^{(k)}E_{i}^
			{\top}\right]\right\|_{\text{op}}=\sum_{i=1}^s\sup_{u\in \mathcal{S}^{d_{1}-1},v \in \mathcal{S}^{d_{2}-1}}\mathbb{E}\left[x_{i(k)}^2x_{i(l)}u^{\top}\Theta_{\star}^{(k)}vv^{\top}E_{i}^{\top}u\right]\\&\leq \sum_{i=1}^s\left(\left\|\Theta_{\star}^{(k)}\right\|_{\text{op}}\sup_{\substack{u\in \mathcal{S}^{d_{1}-1}\\v \in \mathcal{S}^{d_{2}-1}}}\sqrt{\mathbb{E}\left[x_{i(k)}^4\right]\mathbb{E}\left[x_{i(l)}^2\left(uE_{i}v^{\top}\right)^2\right]}\right)\leq 4\kappa_{0}^2I_{1}^{\frac{1}{2}}\sum_{i=1}^s\left\|\Theta_{\star}^{(k)}\right\|_{\text{op}}\leq 4s\kappa_{0}^3\sqrt{\frac{2\ell M}{\ell-1}}R.
		\end{aligned}
		$$
		Therefore, $\left\|\mathbb{E}\left[x_{i(k)}^2Y_{i}Y_{i}^{\top}\right]\right\|_{\text{op}}\leq \kappa_0^2\left(\frac{2\ell M}{\ell-1}+16s^2\kappa_{0}^2R^2+8s\kappa_{0}\sqrt{\frac{2\ell M}{\ell-1}}R\right)=\kappa_{0}^2\left(4\kappa_{0}sR+\sqrt{\frac{2\ell M}{\ell-1}}\right)^2$. By the same way, $\left\|\mathbb{E}\left[x_{i(k)}^2Y_{i}^{\top}Y_{i}\right]\right\|_{\text{op}}\leq \kappa_{0}^2\left(4\kappa_{0}sR+\sqrt{\frac{2\ell M}{\ell-1}}\right)^2$. From Corollary 3.1 in Minsker (2018)\cite{Minsker1} with shrinkage function $\psi_{\tau_{k}}$, we obtain that
		$$\operatorname{P}\left(\left\|\widehat{\Sigma}_{x_{(k)} Y}-\Sigma_{x_{(k)}Y}\right\|_{\text{op}} \geq t\right) \leq 2(d_{1}+d_{2}) \exp \left(-n t/\tau_{k}+\frac{n\sigma_{k}^2}{2\tau_{k}^2}\right). $$
		Further choosing $\tau_{k}\asymp \sigma_{k}\sqrt{n/\log(d_{1}+d_{2})}$ and $t=\sigma\delta\sqrt{\frac{\log(d_{1}+d_{2})}{n}},$ under the condition of $s\asymp \sqrt{d_{1}+d_{2}}$, with probability at least $1-2\left(d_1+d_2\right)^{2-\delta}$, we have that
		\begin{equation}\label{eq16}
			\left\|\widehat{\Sigma}_{x_{(k)} Y}-\Sigma_{x_{(k)}Y}\right\|_{\text{op}} \leq  C\sigma\delta \sqrt{\frac{\log \left(d_1+d_2\right)}{n}}\leq  CR\delta \sqrt{\frac{M\left(d_1+d_2\right) \log \left(d_1+d_2\right)}{n}}.
		\end{equation}
		On the other hand,
		\begin{equation}\label{eq17}
			\begin{gathered}
				T_{2}\leq\sum_{j=1}^s\left(\left|\left(\widehat{\Sigma}_{XX}-\Sigma_{XX}\right)_{i,j}\right| \left\|\Theta_\star^{(j)}\right\|_{\text{op}}\right) \leq \left(\sum_{j=1}^s \left(\widehat{\Sigma}_{XX}-\Sigma_{XX}\right)_{i,j}^2\right)^{1 / 2}\left(\sum_{j=1}^s\left\|\Theta_\star^{(j)}\right\|_{\text{op}}^2\right)^{1 / 2}\\ \leq  R\sqrt{s}\left\|\widehat{\Sigma}_{XX}-\Sigma_{XX}\right\|_{\text{op}}\lesssim  R K^2 \left\|\Sigma_{XX}\right\|_{\text{op}}\sqrt{\frac{s^2}{n}}\lesssim
				R K^2\left\|\Sigma_{XX}\right\|_{\text{op}}\sqrt{ (d_{1}+d_{2}) /n}.
			\end{gathered}
		\end{equation}
		By the proof of Theorem 3 in Chen at al. (2023)\cite{Chen2}, we obtain that $\sum_{i=1}^s\left\|\Theta_\star^{(i)}-\widehat{\Theta}^{(i)}\right\|_{\star}\leq \sqrt{r}\left\|\Theta_\star^{(i)}-\widehat{\Theta}^{(i)}\right\|_{\text{op}}$. Therefore, combining (\ref{eq14})-(\ref{eq17}) and choosing $\lambda\asymp R\delta\sqrt{\frac{M\left(d_1+d_2\right) \log \left(d_1+d_2\right)}{n}}$, with probability at least $1-2s(d_{1}+d_{2})^{2-\delta}-2s^{-\delta}$, 
		$$
		\begin{gathered}
			\frac{\kappa_0}{2}\|\Delta\|_F^2\lesssim R\delta \sqrt{\frac{M\left(d_1+d_2\right) \log \left(d_1+d_2\right)}{n}}\cdot \sum_{i=1}^s\left\|\widehat{\Theta}^{(i)}-\Theta_\star^{(i)}\right\|_{\star}+\lambda \cdot \sum_{i=1}^s\left(\left\|\Theta_\star^{(i)}\right\|_{\star}-\left\|\widehat{\Theta}^{(i)}\right\|_{\star}\right)\\
			\lesssim R\delta \sqrt{\frac{Mr\left(d_1+d_2\right) \log \left(d_1+d_2\right)}{n}}\cdot \sum_{i=1}^s\left\|\widehat{\Theta}^{(i)}-\Theta_\star^{(i)}\right\|_{\star}=R\delta \sqrt{\frac{Mr\left(d_1+d_2\right) \log \left(d_1+d_2\right)}{n}}\left\|\Delta\right\|_{F}.
		\end{gathered}$$
		Hence, $\operatorname{P}\left(\left\|\widehat{\Theta}-\Theta_{\star}\right\|_{F}\lesssim R\delta \sqrt{\frac{Mr\left(d_1+d_2\right) \log \left(d_1+d_2\right)}{n}}\right)\geq 1-2(d_{1}+d_{2})^{5/2-\delta}-2(d_{1}+d_{2})^{-\delta/2}$.
	\end{proof}
	
	\begin{lemma}\label{lemma2}
		(Theorem 6.1.1 in Tropp (2015)\cite{Tropp}) 
		Consider $n$ independent random matrices $\{S_{i}\in \mathbb{R}^{d_{1}\times d_{2}}: \mathbb{E}[S_{i}]=0_{d_{1}\times d_{2}}, \|S_{i}\|_{\text{op}}\leq L, a.s. \}_{i\in [n]}$. For $\forall t>0$, 
		$$
		\operatorname{P}\left(\left\|\frac{1}{n}\sum_{i=1}^n S_{i}\right\|_{\text{op}}\geq t\right)\leq (d_{1}+d_{2})\exp\left(\frac{-nt^2/2}{\nu+Lt/3}\right)
		$$
		where $\nu:= \left\|\frac{1}{n}\sum_{i=1}^n\mathbb{E}[S_{i}S_{i}^{\top}]\right\|_{\text{op}}\vee \left\|\frac{1}{n}\sum_{i=1}^n\mathbb{E}[S_{i}^{\top}S_{i}]\right\|_{\text{op}}$.
	\end{lemma}
	\section{Numerical simulations}\label{3}
	In this section, we perform numerical simulations of the theoretical results from the previous section to illustrate the validity of the estimators. The results are based on 200 independent replications.
	\subsection{Verification of Theorem \ref{theorem1}}
	To facilitate the simulation, let $d_{1}=d_{2}=d$ and design $\Theta_{\star}=V_{7}V_{7}^\top$, where $V_{7}$ is the top 7 eigenvectors of $\frac{1}{n}\sum_{i=1}^{100}Z_{i}Z_{i}^{\top}$ and $\{Z_{ i}\}_{i=1}^{100}$ is 100 i.i.d. $d$-dimensional standard Gaussian random vectors. For the two cases of Theorem \ref{theorem1}, we consider the following two types of sample distributions.\par
	(a) Each component of the covariate $X_{i}$ is i.i.d. sampled in $\mathcal{N}(0,1)$. Each component of the random error term $\epsilon_{i}$ obeys $t_{2.1}/5$ distribution; \par
	(b) The covariate $X_{i}\sim T_{d}(0_{d\times 1},I_{d},6)$  and the random noise term $\epsilon_{i}\sim T_{d}(0_{d\times 1},I_{d},4.1)$.
	
	The experimental results of (a) are shown in Figure \ref{fig1}, where the dashed line represents the theoretical convergence rate $O(n^{-1/2})$. $\eta_{1}=\eta_{2}=0$ means that no quantization for data. the experimental results on the left panel represents the case under the quantization of $\{(X_{i}, Y_{i})\}_{i=1}^n$, while the right panel represent the case under the quantization of only $\{Y_{i}\}_{i=1}^n$. It can be seen that the estimation error enlarges as $\eta_{1},\eta_{2}$ increases gradually. Meanwhile, for fixed $\eta_{1}$ and $\eta_{2}$, each line is almost parallel to the imaginary line, which verifies the conclusion of Theorem \ref{theorem1}. \par
	The experimental result of (b) is shown in Figure \ref{fig2}, which is consistent with the features of Figure \ref{fig1}. Note that for the choice of the parameter $\tau$, we use the adaptive equation of Li et al. (2021)\cite{Li} and Ke et al. (2019)\cite{Ke} to determine its level,
	$$\left\|\frac{1}{\tau^{4}} \sum_{i=1}^{n}\left(\left\|X_{i}\right\|_{2}^{2} \bigwedge \tau^2\right)^{2} \frac{X_{i} X_{i}^{\top}}{\left\|X_{i}\right\|_{2}^{2}}\right\|_{\text{op}}=\log(2d)+\log(n).$$
	Therefore, $\tau \asymp\sqrt[4]{\frac{ndM}{\log(2d)+\log(n)}}$. The conclusion of (b) in Theorem \ref{theorem1} becomes $\left\|\widehat{\Theta}-\Theta_{\star}\right\|_{F}\lesssim \sqrt{\frac{rd(\log(d)+\log(n))}{n}}\asymp_{r ,d} \sqrt{\frac{\log(n)}{n}}$ with overwhelming probability. The dashed line in Figure \ref{fig2} represents $O\left(\sqrt{\log(n)/n}\right)$.
	
	\begin{figure}[H]
		\centering
		\pgfplotsset{every axis x label/.append style={at={(0.5,0.05)}}, every axis y label/.append style={at={(0.08,0.5)}}}
		\begin{minipage}{0.4\linewidth}
			\begin{tikzpicture}
				\begin{axis}[title={(1) $d=80, t_{2.1}/5$ noise}, title style={font=\tiny},font=\tiny,xlabel={$\log(n)$}, ylabel={$\log\|\widehat{\Theta}-\Theta_{\star}\|_{F}$},height=5.6cm, width=6.5cm]
					\addplot[
					color=blue,
					mark=+,
					]
					coordinates {
						(6.620073,-0.2341787)(6.907755, -0.3673022)(7.313220, -0.5518211)(7.600902, -0.6852131)(7.824046, -0.7836071)(8.070906,-0.8964641)
					};
					\addplot[
					color=red,
					mark=triangle,
					]
					coordinates {
						(6.620073,-0.2119177)(6.907755,-0.348616)(7.313220,-0.531895)(7.600902, -0.6628703)(7.824046, -0.7599927)(8.070906,-0.8717246) };
					\addplot[
					color=green,
					mark=square,
					]
					coordinates {
						(6.620073,-0.1564213)(6.907755,-0.2857797)(7.313220,-0.471012)(7.600902,-0.5996385)(7.824046, -0.6983708)(8.070906,-0.8031264) };
					\addplot[
					color=orange,
					mark=o
					]
					coordinates {
						(6.620073,-0.05795987)(6.907755,-0.1838115)(7.313220,-0.370616)(7.600902,-0.4944374)(7.824046,-0.5983702)(8.070906,-0.6977806)
					};
					\addplot [
					domain=6.6:7.8,
					samples=100,
					color=black,
					densely dashed]
					{-0.5*x+3};
					
					\legend{$\eta_{1}\text{,}\eta_{2}=0$, $\eta_{1}\text{,}\eta_{2}=0.2$, $\eta_{1}\text{,}\eta_{2}=0.4$, $\eta_{1}\text{,}\eta_{2}=0.6$}
				\end{axis}
			\end{tikzpicture}
		\end{minipage}
		\begin{minipage}{0.4\linewidth}
			\begin{tikzpicture}
				\begin{axis}[title={(2) $d=50, t_{2.1}/5$ noise}, title style={font=\tiny},font=\tiny,xlabel={$\log(n)$}, ylabel={$\log\|\widehat{\Theta}-\Theta_{\star}\|_{F}$},height=5.6cm, width=6.5cm]
					\addplot[
					color=blue,
					mark=+,
					]
					coordinates {
						(6.907755, -0.6617404)(7.313220, -0.8558635)(7.600902, -0.9801574)(7.824046, -1.0829448 )(8.070906,-1.197041)(8.29405,-1.290008 )
					};
					\addplot[
					color=red,
					mark=triangle,
					]
					coordinates {
						(6.907755, -0.6520427)(7.313220, -0.8423659)(7.600902, -0.9695743)(7.824046, -1.0707237)(8.070906, -1.184060)(8.29405,-1.281405) };
					\addplot[
					color=green,
					mark=square,
					]
					coordinates {
						(6.907755, -0.6158401)(7.313220, -0.8069973)(7.600902,-0.9360040)(7.824046, -1.0351001)(8.070906,-1.146687)(8.29405, -1.247112) };
					\addplot[
					color=orange,
					mark=o
					]
					coordinates {
						(6.907755,-0.5524442)(7.313220, -0.7474034)(7.600902, -0.8784678)(7.824046,-0.9819397)(8.070906, -1.087757)
						(8.29405, -1.191038)};
					\addplot[
					color=violet,
					mark=x
					]
					coordinates {
						(6.907755,-0.4818295)(7.313220, -0.6727380)(7.600902, -0.8019017)(7.824046, -0.9031760)(8.070906, -1.0105214)(8.29405, -1.1173149)};  
					\addplot [
					domain=6.9:8,
					samples=100,
					color=black,
					densely dashed]
					{-0.5*x+2.7};
					\legend{$\eta_{2}=0$, $\eta_{2}=0.2$, $\eta_{2}=0.4$, $\eta_{2}=0.6$, $\eta_{2}=0.8$}
				\end{axis}
			\end{tikzpicture}
		\end{minipage}
		\caption{(1): bounded moment response under complete quantization; (2) bounded moment response under partial quantization. The $x$-axis and $y$-axis represent logarithmic sample size and $\log\left\|\widehat{\Theta}-\Theta^{*}\right\|_{F}$.}
		\label{fig1}
	\end{figure}
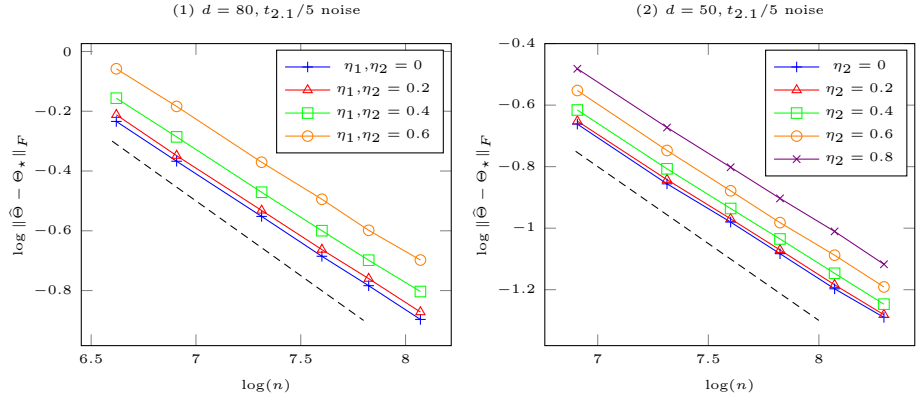
	\begin{figure}[H]
		\centering
		\pgfplotsset{every axis x label/.append style={at={(0.5,0.05)}}, every axis y label/.append style={at={(0.08,0.5)}}}
		\begin{minipage}{0.4\linewidth}
			\begin{tikzpicture}
				\begin{axis}[title={(1) $d=50,T_{6}$ covariate, $T_{4.1}$ noise}, title style={font=\tiny},font=\tiny,xlabel={$\log(n)$}, ylabel={$\log\|\widehat{\Theta}-\Theta_{\star}\|_{F}$},height=5.6cm, width=6.5cm]
					\addplot[
					color=blue,
					mark=+,
					]
					coordinates { 
						(6.907755, 0.1332580)(7.313220, -0.029276981)(7.600902, -0.14734457)(7.824046, -0.2416567)(8.070906,-0.3450322)(8.29405,-0.4452155)
					};
					\addplot[
					color=red,
					mark=triangle,
					]
					coordinates {
						(6.907755, 0.1424376)(7.313220, -0.022209955)(7.600902, -0.14013570)(7.824046,  -0.2331091)(8.070906, -0.3395635)(8.29405,-0.4340184) };
					\addplot[
					color=green,
					mark=square,
					]
					coordinates {
						(6.907755, 0.1665537)(7.313220, 0.001250595)(7.600902,-0.11670876)(7.824046,  -0.2060363)(8.070906,-0.3172059)(8.29405, -0.4105111) };
					
					\addplot[
					color=orange,
					mark=o
					]
					coordinates {
						(6.907755,0.2051726)(7.313220, 0.040435593)(7.600902, -0.07482272)(7.824046,-0.1730983)(8.070906, -0.277922)(8.29405, -0.3786748)};
					\addplot [
					domain=6.9:8,
					samples=100,
					color=black,
					densely dashed]
					{-0.5*x+0.5*ln(x)+2.53};
					\legend{$\eta_{1}\text{,}\eta_{2}=0$, $\eta_{1}\text{,}\eta_{2}=0.2$, $\eta_{1}\text{,}\eta_{2}=0.4$, $\eta_{1}\text{,}\eta_{2}=0.6$}
				\end{axis}
			\end{tikzpicture}
		\end{minipage}
		\begin{minipage}{0.4\linewidth}
			\begin{tikzpicture}
				
				\begin{axis}[title={(2) $d=60, T_{6}$ covariate, $T_{4.1}$ noise}, title style={font=\tiny},font=\tiny,xlabel={$\log(n)$}, ylabel={$\log\|\widehat{\Theta}-\Theta_{\star}\|_{F}$},height=5.6cm, width=6.5cm]
					\addplot[
					color=blue,
					mark=+,
					]
					coordinates { 
						(7.313220, 0.1019765)(7.600902, -0.0126141626)(7.824046, -0.09961785)(8.070906,-0.2037297)(8.29405,-0.2954001)(8.517193, -0.3914163)
					};
					\addplot[
					color=red,
					mark=triangle,
					]
					coordinates {
						(7.313220, 0.1033506)(7.600902, -0.0072881)(7.824046, -0.099921)(8.070906,-0.200664)(8.29405,-0.29342)(8.517193, -0.38773) };
					\addplot[
					color=green,
					mark=square,
					]
					coordinates {
						(7.313220, 0.1107)(7.600902, 0.000551)(7.824046,-0.091652)(8.070906, -0.1929)(8.29405,-0.28467)(8.517193, -0.381225) };
					
					\addplot[
					color=orange,
					mark=o
					]
					coordinates {
						(7.313220,0.131036)(7.600902, 0.017919)(7.824046, -0.074388)(8.070906,-0.17842)(8.29405, -0.270653)
						(8.517193, -0.367202)};
					\addplot[
					color=violet,
					mark=x
					]
					coordinates {
						(7.313220,0.1536517)(7.600902, 0.03704448)(7.824046, -0.055469)(8.070906,-0.157288)(8.29405, -0.2530)(8.517193, -0.34768)};  
					\addplot [
					domain=7.32:8.26,
					samples=100,
					color=black,
					densely dashed]
					{-0.5*x+0.5*ln(x)+2.7};
					\legend{$\eta_{2}=0$, $\eta_{2}=0.2$, $\eta_{2}=0.4$, $\eta_{2}=0.6$,$\eta_{2}=0.8$}
				\end{axis}
			\end{tikzpicture}
		\end{minipage}
		\caption{(1): bounded moment design under complete quantization; (2) bounded moment design under partial quantization. The $x$-axis and $y$-axis represent logarithmic sample size and $\log\left\|\widehat{\Theta}-\Theta^{*}\right\|_{F}$.}
		\label{fig2}
	\end{figure}
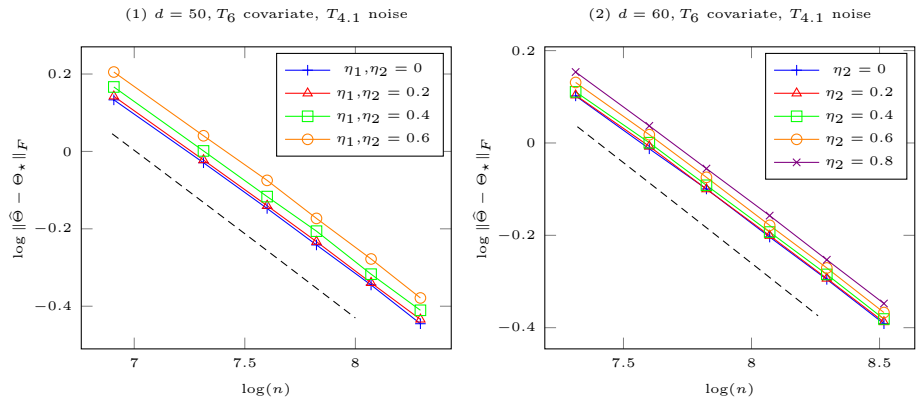
	\subsection{Verification of Theorem \ref{theorem2}}
	To simulate low-rank linear model with matrix-type responses, $s$ low-rank parameter matrices $\left\{\Theta_{\star}^{(k)}\right\}_{k\in [s]}$ are generated as follows: let $d_{1}=d_{2}$ and for any $k\in [s]$, we generate $\Theta_{1}^{(k) } \in \mathbb{R}^{d_1 \times r}$ and $\Theta_{2}^{(k)} \in \mathbb{R}^{d_2\times r}$, where each term of $\Theta_{1}^{(k)} $ and $\Theta_{2}^{(k)}$ i.i.d. follows the standard normal distribution. Then we get $\Theta^{(k)}_{\star}=\Theta_{1}^{(k)}\Theta_{2}^{(k)}/\left\|\Theta_{1}^{(k)}\Theta_{2}^{(k)}\right\|_{F}$. The covariate $X\sim \mathcal{N}(0_{s\times 1}, I_{s})$. For the random error term $E$, we consider two cases: (a) $\text{vec}(E) \sim T(0_{d_{1}d_{2}\times 1}, I_{d_{1}\star d_{2}}, 2.1)/10$; (b) $E \sim \left(ZZ^{\top}-5/3I_{d_{1}}\right)/10$, where $Z\sim T(0_{d_{1}\times 1}, I_{d_{1}},5)$. \par
	Since $\tau_{k}\asymp\sigma_{k}\sqrt{n/(\log(d_{1}+d_{2})+\log(n))}$ and $\sigma_{k}^2=\left\|\mathbb{E}\left[x_{i(k)}^2Y_{i}Y_{i}^{\top}\right]\right\|_{\text{op}}\vee \left\|\mathbb{E}\left[x_{i(k)}^2Y_{i}^{\top}Y_{i}\right]\right\|_{\text{op}}=\left\|\mathbb{E}\left[\mathcal{F}\left(x_{i(k)}Y_{i}\right)^2\right]\right\|_{\text{op}}$, We select $\tau_{k}=c\sqrt{\left\|\mathbb{E}\left[\mathcal{F}\left(x_{i(k)}Y_{i}\right)^2\right]\right\|_{\text{op}}n/\left(\log(d_{1}+d_{2})+\log(n)\right)}$, given the constant $c>0$. $\mathbb{E}\left[\mathcal{F}\left(x_{i(k)}Y_{i}\right)^2\right]$ can be estimated by the truncated robust estimator $\frac{1}{n} \sum_{i=1}^n\psi_{\tau_{k}}(\mathcal{F}(x_{i(k)}Y_{i}))^2$, because the heavy tail of $\{Y_{i}\}_{i=1}^n$ leads the traditional moment estimator $\frac{1}{n}\sum_{i=1}^n \mathcal{F}(x_{i(k)}Y_{i})^2$ to overestimate the true value. Therefore, for $\forall k\in [s]$, solve the following $s$ equations to obtain the parameters $\{\tau_{k}\}_{k\in [s]}$:
	$$\left\|\frac{1}{\tau_{k}^2} \sum_{i=1}^n\psi_{\tau_{k}}(\mathcal{F}(x_{i(k)}Y_{i}))^2\right\|_{\text{op}}=4\log (d_{1}+d_{2})+4\log (n),$$
	where $c=1/4$.\par
	The results are showed by Figure \ref{fig3}, where ``Robust'' represents the proposed estimator and ``Standard'' stands for the the traditional least squares estimator. For each case, our robust estimator is better than the original regularized least squares estimate which is unstable and has significant fluctuations in estimation errors.
	\begin{figure}
		\centering
		\begin{minipage}{0.49\linewidth}
			\includegraphics[width=2.95in,height=2.6in]{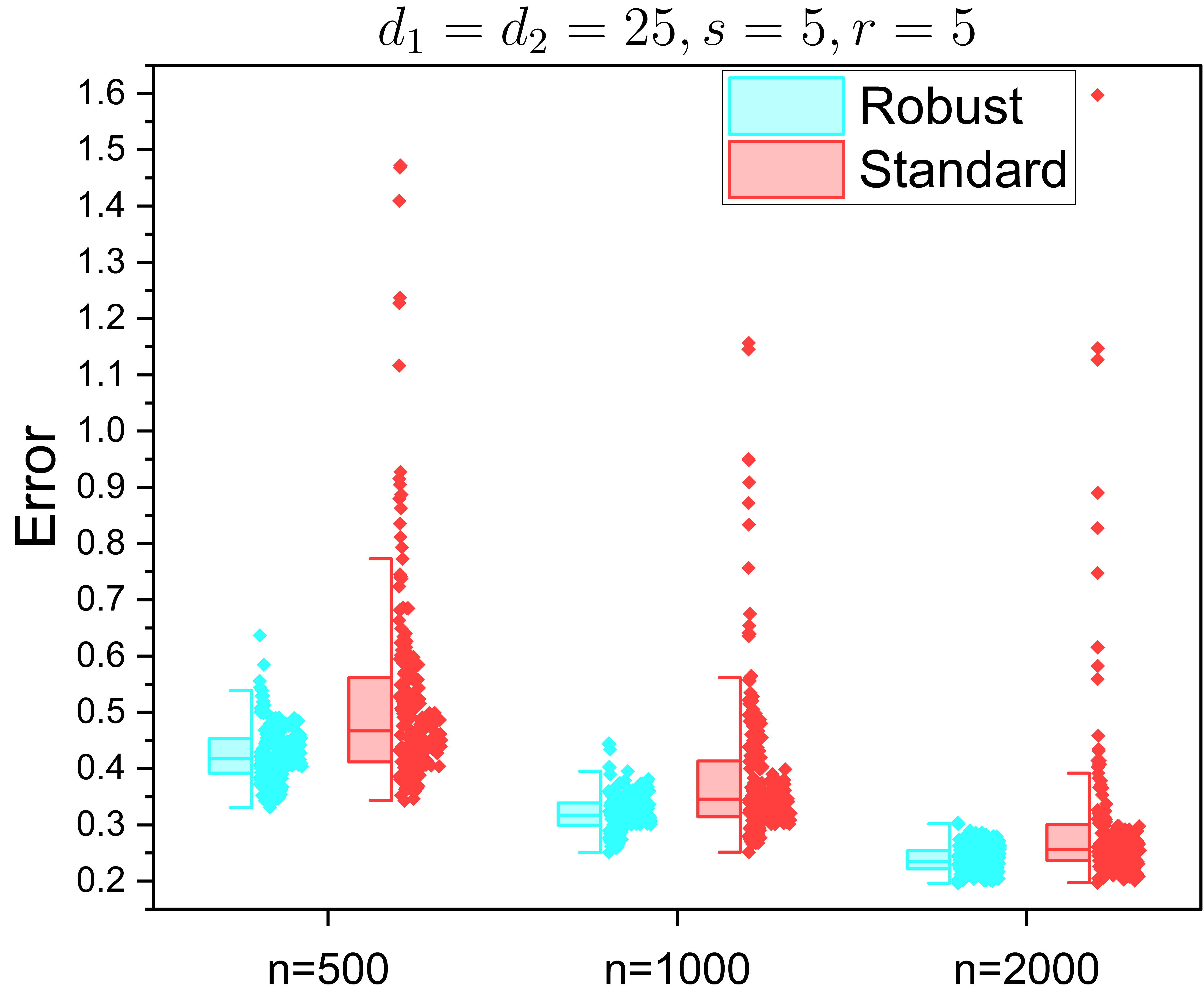}
		\end{minipage}
		\begin{minipage}{0.49\linewidth}
			\includegraphics[width=2.95in,height=2.6in]{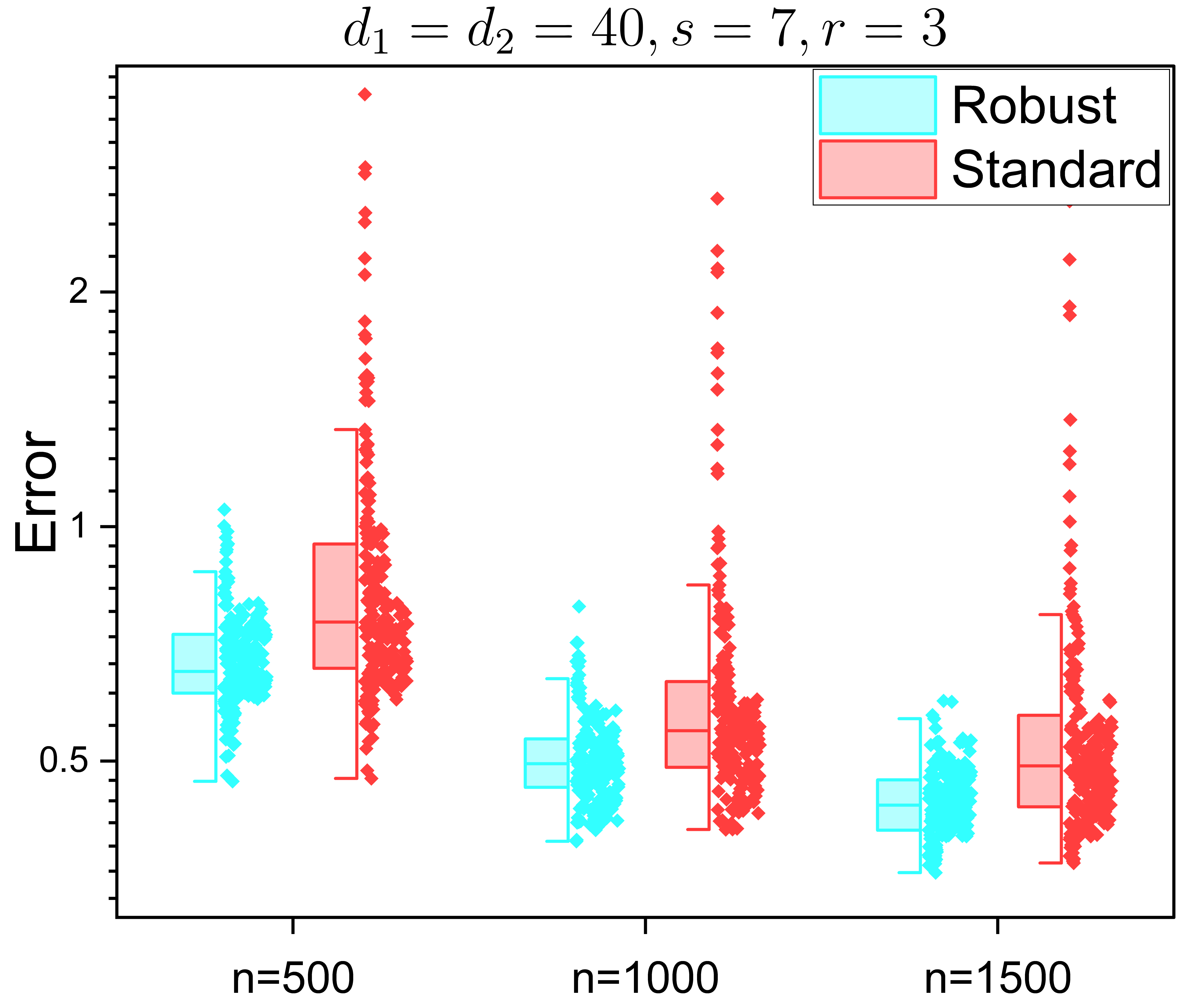}
			
		\end{minipage}
		
		\centering
		\begin{minipage}{0.49\linewidth}
			\includegraphics[width=2.95in,height=2.6in]{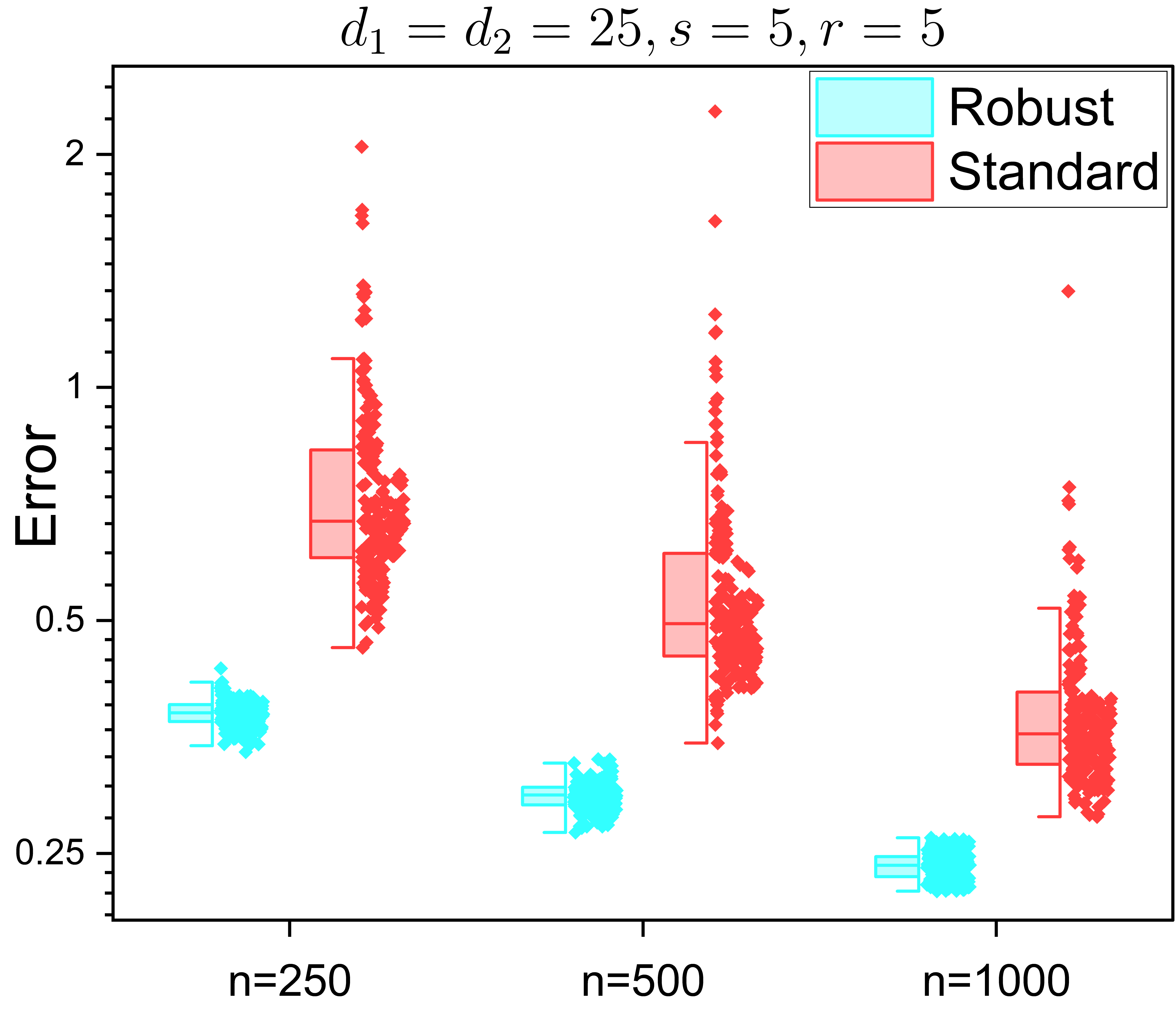}
		\end{minipage}
		\begin{minipage}{0.49\linewidth}
			\includegraphics[width=2.95in,height=2.6in]{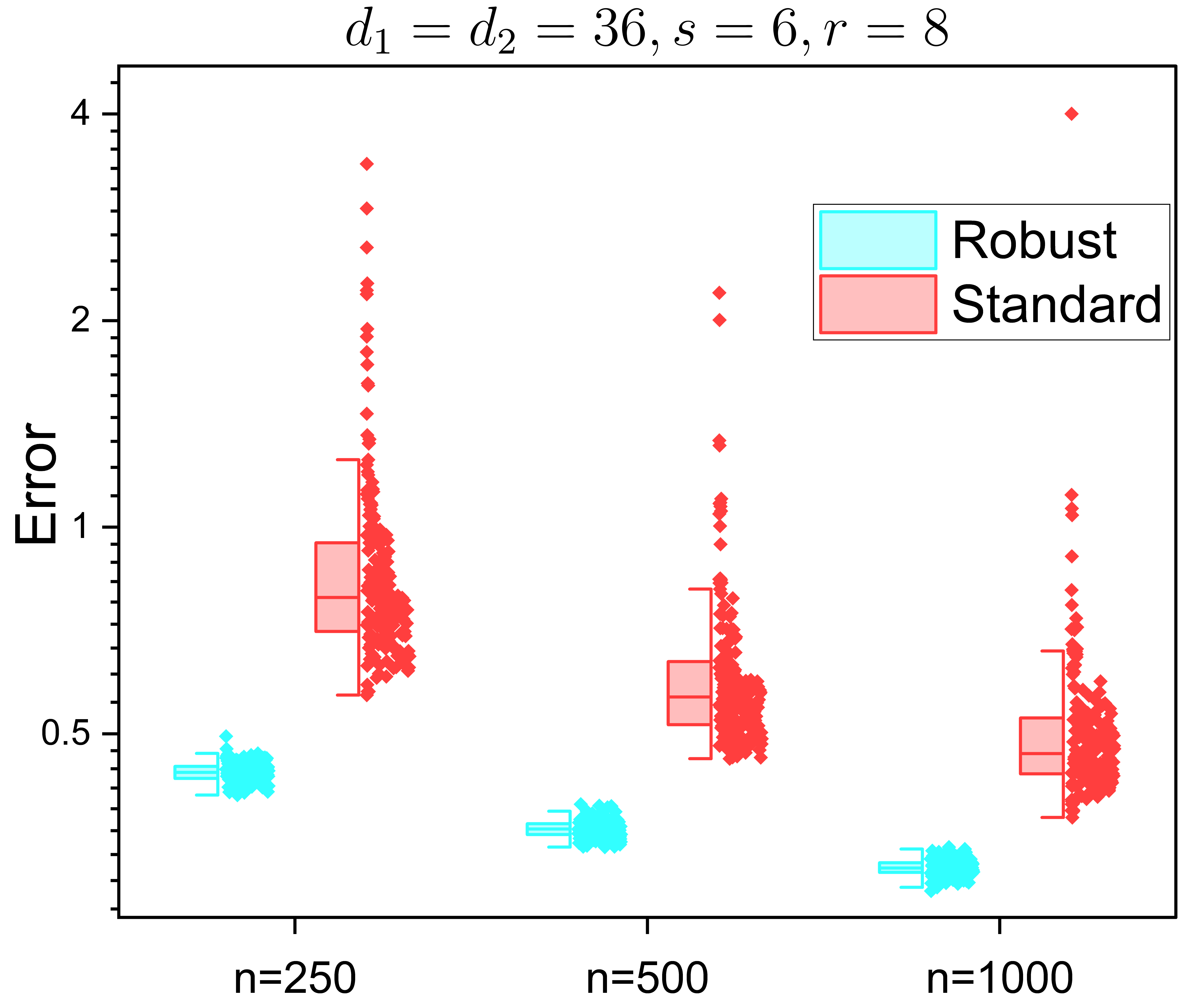}
			\caption{Statistical errors $\left\|\widehat{\Theta}-\Theta_{\star}\right\|_{F}$ v.s. the sample size $n$. The first row corresponds to case (a), while the second row corresponds to case (b). }
			\label{fig3}
		\end{minipage}
	\end{figure}

	In the next simulation, we use four $43\times 53$ dimensional $0$-$1$ matrices (Kong et al (2020)\cite{Kong}, Figure1) as the parameter matrix $\left\{\Theta_{\star}^{(k)}\right\}_{k\in [4]}$ of the model (\ref{eq1}), as shown in the image of the first row of Figure \ref{fig4}. the sample size $n=500$ and $X\sim \mathcal{N}(0_{4\times 1}, I_{4})$. For random noise term $E$, we consider the following cases:\par 
	(a) $E=Z_{1}Z_{2}^{\top}$ where $Z_{1}\sim T(0_{43\times 1}, I_{43}, 3)$ and $Z_{2}\sim T(0_{53\times 1}, I_{53}, 3)$. $Z_{1}$ and $Z_{2}$ are independent;\par
	(b) $E=[Z_{1}, Z_{2}, \cdots, Z_{53}]$ where $\{Z_{i}\}_{i=1}^{53}\stackrel{i.i.d.}{\sim} T(0_{43\times 1}, I_{43}, 2.1).$\par
	The results are summarized by Table \ref{tab1} and display that our robust estimator has much better performance than the baseline in both average estimation error and standard deviation. We randomly select a dataset from 200 reduplicate experiments and plot two estimates of $\left\{\Theta_{\star}^{(k)}\right\}_{k\in [4]}$. The images are depicted in Figure \ref{fig4} and illustrate that our robust estimator outperforms the
	traditional least squares estimate.
	\begin{figure}
		\centering
		\begin{minipage}{0.244\linewidth}
			\includegraphics[width=1.5in,height=1.3in]{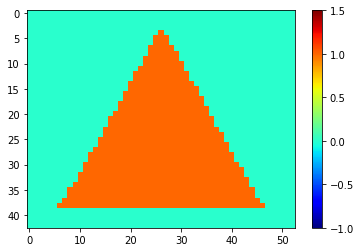}
		\end{minipage}
		\begin{minipage}{0.244\linewidth}
			\includegraphics[width=1.5in,height=1.3in]{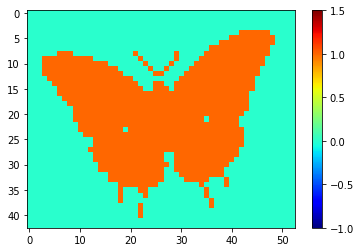}
		\end{minipage}
		\centering
		\begin{minipage}{0.244\linewidth}
			\includegraphics[width=1.5in,height=1.3in]{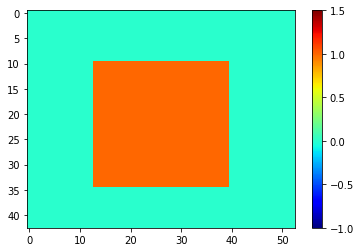}
		\end{minipage}
		\begin{minipage}{0.244\linewidth}
			\includegraphics[width=1.5in,height=1.3in]{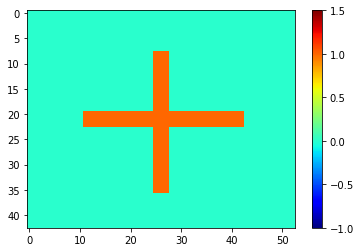}
		\end{minipage}
		
		\centering
		\begin{minipage}{0.244\linewidth}
			\includegraphics[width=1.5in,height=1.3in]{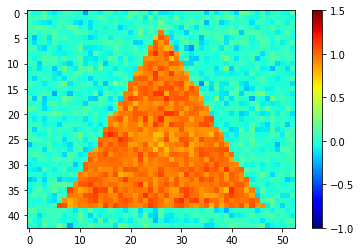}
		\end{minipage}
		\begin{minipage}{0.244\linewidth}
			\includegraphics[width=1.5in,height=1.3in]{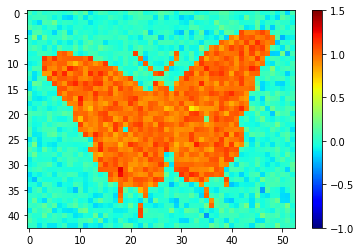}
		\end{minipage}
		\centering
		\begin{minipage}{0.244\linewidth}
			\includegraphics[width=1.5in,height=1.3in]{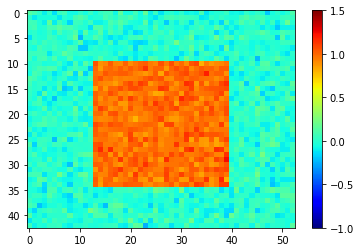}
		\end{minipage}
		\begin{minipage}{0.244\linewidth}
			\includegraphics[width=1.5in,height=1.3in]{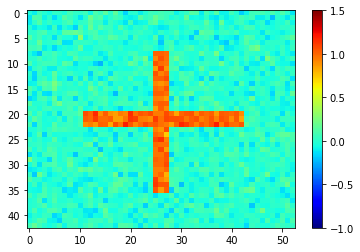}
			
		\end{minipage}
		\centering
		\begin{minipage}{0.244\linewidth}
			\includegraphics[width=1.5in,height=1.3in]{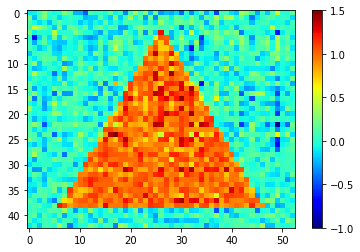}
		\end{minipage}
		\begin{minipage}{0.244\linewidth}
			\includegraphics[width=1.5in,height=1.3in]{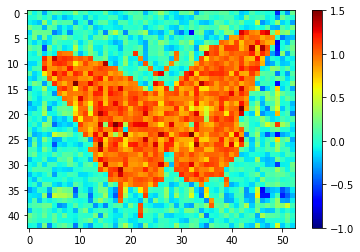}
		\end{minipage}
		\centering
		\begin{minipage}{0.244\linewidth}
			\includegraphics[width=1.5in,height=1.3in]{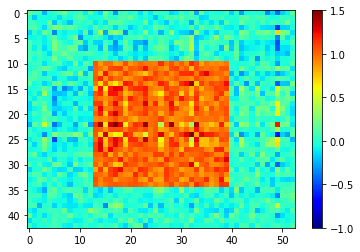}
		\end{minipage}
		\begin{minipage}{0.244\linewidth}
			\includegraphics[width=1.5in,height=1.3in]{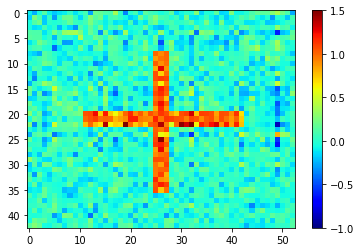}
		\end{minipage}
		
		\centering
		\begin{minipage}{0.244\linewidth}
			\includegraphics[width=1.5in,height=1.3in]{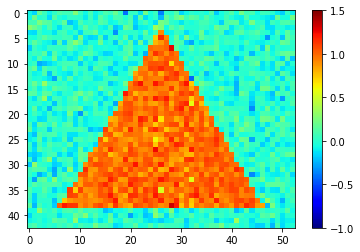}
		\end{minipage}
		\begin{minipage}{0.244\linewidth}
			\includegraphics[width=1.5in,height=1.3in]{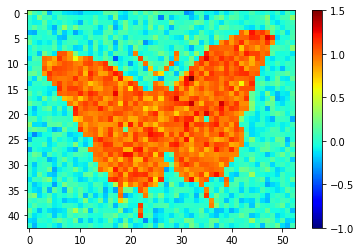}
		\end{minipage}
		\centering
		\begin{minipage}{0.244\linewidth}
			\includegraphics[width=1.5in,height=1.3in]{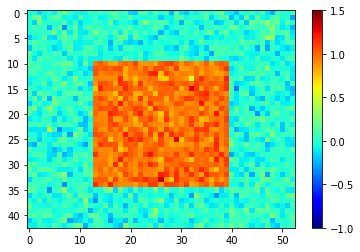}
		\end{minipage}
		\begin{minipage}{0.244\linewidth}
			\includegraphics[width=1.5in,height=1.3in]{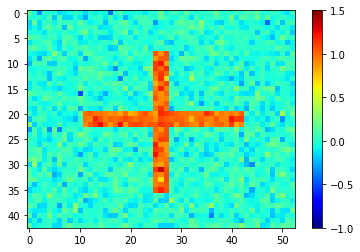}
			
		\end{minipage}
		\centering
		\begin{minipage}{0.244\linewidth}
			\includegraphics[width=1.5in,height=1.3in]{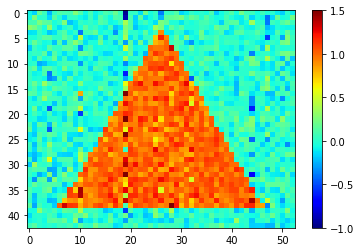}
		\end{minipage}
		\begin{minipage}{0.244\linewidth}
			\includegraphics[width=1.5in,height=1.3in]{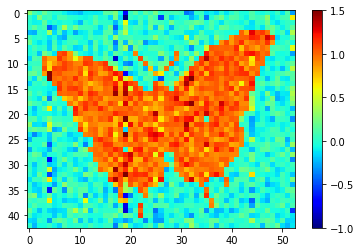}
		\end{minipage}
		\centering
		\begin{minipage}{0.244\linewidth}
			\includegraphics[width=1.5in,height=1.3in]{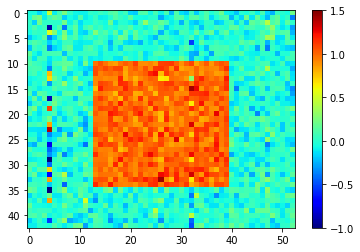}
		\end{minipage}
		\begin{minipage}{0.244\linewidth}
			\includegraphics[width=1.5in,height=1.3in]{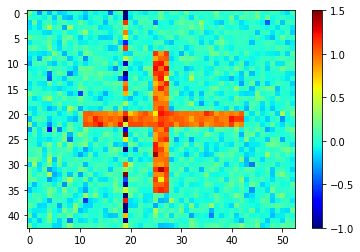}
			\caption{The first row is true images; The second and fourth rows are images reconstructed via the proposed estimator under case (a) and (b), respectively; The third and fifth rows are images reconstructed via the original least squares method under case (a) and (b).} 
			\label{fig4}
		\end{minipage}
	\end{figure}
	
	\begin{table*}[tp]
		\renewcommand\arraystretch{1.7}
		\centering
		\fontsize{10}{10}\selectfont
		\caption{Comparison of Frobenius norm estimation errors between two approaches. }  
		\label{tab1}
		\begin{tabular}{|c|c|c|c|c|c|c|}
			\hline
			Order & Methods & $\Theta_{\star}^{(1)}$ & $\Theta_{\star}^{(2)}$ & $\Theta_{\star}^{(3)}$ & $\Theta_{\star}^{(4)}$ & $\Theta_{\star}$  \\
			\hline 
			\multirow{2}{*}{(a)}& Standard  & $5.7673_{(1.30)}$ & $5.9638_{(1.72)}$ &  $5.9060_{(1.87)}$ &  $5.7308_{(1.45)}$ & $11.8539_{(2.51)}$   \\
			\cdashline{2-7}[1.5pt/2pt]
			&{\bf Robust}  & $\bf {3.9041_{(0.20)}}$ & $\bf  {3.8574_{(0.20)}}$ &  $\bf {3.7779_{(0.19)}}$ &   $\bf {3.4920_{(0.18)}}$ &  $\bf {7.5275_{(0.26)}}$    \\
			\hline
			\multirow{2}{*}{(b)}& Standard  & $6.4992_{(1.80)}$ &  $6.6159_{(2.63)}$ &  $6.2391_{(1.49)}$ &  $6.5053_{(2.14)}$ & $13.0799_{(3.63)}$ \\
			\cdashline{2-7}[1.5pt/2pt]
			&{\bf Robust}  & $\bf {5.1581_{(0.20)}}$ & $\bf {5.1335_{(0.17)}}$ & $\bf{ 5.0668_{(0.23)}}$ & $\bf {4.9595_{(0.22)}}$ & $\bf {10.1652_{(0.26)}}$  \\
			\hline
		\end{tabular}
	\end{table*}
	\section{Discussions}\label{4}
	There are some shortcomings in this paper that need further improvement and research. For example, for linear model with the matrix-type response, we only study the heavy-tailed scenario and it is our future research direction whether we can extend it to the quantization case. On the other hand, the framework of this paper can be directly extended to the one-bit quantization case with sub-Gaussian data (Chen et al. (2023)\cite{Chen1}; Dirksen et al. (2022)\cite{Dirksen2}). Due to the limitation in space and time, we will leave this idea in future work.
	
\end{document}